\documentclass[11pt]{article}
\usepackage{mathrsfs}
\usepackage{mathrsfs}
\usepackage{mathrsfs}
\usepackage{diagbox}
\usepackage{cite}
\usepackage{amsfonts}
\usepackage{amssymb}
\usepackage{amsmath}
\usepackage{amsthm}
\usepackage{dsfont}
\usepackage[colorlinks=white,
            linkcolor=blue,
            anchorcolor=blue,
            citecolor=blue
            ]{hyperref}
\theoremstyle{plain}
\newtheorem{thm}{Theorem}[section]

\newtheorem{lem}[thm]{Lemma}

\theoremstyle{definition}
\newtheorem{defn}{Definition}[section]

\theoremstyle{remark}

\begin{document}
\title{{\Large\bf {
Optimal control of quantum stochastic systems in fermion fields: The Pontryagin-type maximum principle $\textrm{(II)}$}}}
\author{{\normalsize  Penghui Wang, Shan Wang,{\thanks{E-mail addresses: phwang@sdu.edu.cn(P.Wang), 202020244@mail.sdu.edu.cn(S.Wang).}}  \,\,\,\,  } \\
{ School of Mathematics, Shandong University,} {\normalsize Jinan, 250100, China} }
\date{}
\maketitle
\begin{minipage}{14cm} {\bf Abstract }
In the present paper, by using the relaxed transposition method \cite{L.Z-2014}, we solve the second-order adjoint equations, corresponding to the optimal control of quantum stochastic systems in fermion fields, which plays the fundamental roles in the study of the Pontryagin-type maximum principle in quantum stochastic optimal control. The second-order adjoint equation is a backward operator-valued quantum stochastic differential equation, which has no definition in the algebra of bounded operators, and the solution derived from the relaxed transposition method makes sense in $W^{*}$-topology.

\end{minipage}
 \maketitle
\numberwithin{equation}{section}
\newtheorem{theorem}{Theorem}[section]
\newtheorem{lemma}[theorem]{Lemma}
\newtheorem{proposition}[theorem]{Proposition}
\newtheorem{corollary}[theorem]{Corollary}
\newtheorem{remark}[theorem]{Remark}


\section{Introduction}\label{Intro}
\indent\indent
The present paper is a continuous work of \cite{W.W}, in which we concentrate on the Pontryagin-type maximum principle for optimal control of quantum stochastic systems in fermion fields. Quantum stochastic control systems arise through the interaction of quantum system with an external field,  which can be expressed in the form of quantum stochastic differential equations (QSDEs, for short) \cite{D.P.R}. The quantum control scenario is a typical representative of quantum optical experiments, as shown in Figure 1.1 of \cite{M.R}. Quantum optical system consists of a clump of atoms,  interacting with an optical field produced by a laser.  We want to control the state of a clump of atoms, for example, to control their collective angular momentum. To observe the atoms, a laser probe field is scattered away from the atoms and the scattered light is measured using a pure difference detector (a cavity can be used to increase the intensity of the interaction between the light and the atoms).     The observed process is fed into a controller, which can control the motion of the atoms through an actuator, e.g., a time-varying magnetic field. 

From the seminal work of von Neumann on the axiomatization of quantum mechanics, quantum probability theory has emerged as a noncommutative generalization of probability theory \cite{S-1}. A general quantum probabilistic theory of quantum stochastic processes was introduced in the work of Accardi, Frigerio, Davies and Lewis etc \cite{AFL,Davies-1,Davies-2,Davies-3,Frigerio,Lewis}.
Quantum probability theory has been rapidly developed by \cite{B,B.H.J,Gardiner-2004,GC-1985,GC-2000,H.L.,P.H,James.2005,James.2012,J.W.W-2024,Meyer,P.Book} since the 1980s, which provides a rigorous basis for studying quantum stochastic calculus and developing a modern formulation of quantum stochastic systems and control theory. In particular, 
Gardiner \cite{Gardiner-2004} gave a statement of the mathematics of quantum Markov processes and developed all quantum analogues of the classical Markovian 
techniques. He explained how to deal with the inputs and outputs in fermion damped quantum systems, and developed fermionic quantum stochastic integration.
Applebaum, Barnrtt, Gough et al. \cite{A.H-1,A.H-2,B.S.W.1,B.S.W.2,B.S.W.3,G.G.J.N} obtained fermion It\^{o}'s formula, developed the fermion quantum stochastic calculus and established fermion filtering theory.

Let us introduce the quantum probability spaces mentioned in this paper.
Let $\mathscr{H}$ be a separable complex Hilbert space. 
The anti-symmetric Fock space over $\mathscr{H}$ \cite{B.S.W.1,P.Book} is defined by
\begin{equation*}
\Lambda(\mathscr{H}):=\bigoplus_{n=0}^\infty\Lambda_n(\mathscr{H}),
\end{equation*}
where $\Lambda_n(\mathscr{H})$ is the Hilbert space anti-symmetric $n$-fold tensor product of $\mathscr{H}$ with itself, and $\Lambda_0(\mathscr{H}):=\mathbb{C}$.
For any $z\in\mathscr{H}$, the creation operator $C(z):\Lambda_n(\mathscr{H})\rightarrow \Lambda_{n+1}(\mathscr{H})$ defined by $v\mapsto \sqrt{n+1}\ z\wedge v$, is a bounded operator on $\Lambda(\mathscr{H})$ with $\|C(z)\|=\|z\|$. 
The annihilation operator $A(z)$ is the adjoint of $C(z)$, i.e., $A(z)=C(z)^*$.
The fermion field $\Psi(z)$ is defined on $\Lambda(\mathscr{H})$ by
\begin{equation}\label{definition Psi}
\Psi(z):=C(z)+A(Jz),
\end{equation}
where the map $J: \mathscr{H}\rightarrow\mathscr{H}$ is a conjugation operator \cite{W.F}, i.e., $J$ is antilinear, antiunitary, and $J^2=1$. 
Then the canonical anti-commutation relation (CAR, for short) holds:
\begin{equation}\label{CAR-Psi}
\{\Psi(z),\Psi(z')\}\equiv\Psi(z)\Psi(z')+\Psi(z')\Psi(z)=2\langle Jz',z\rangle I,\quad z,z'\in \mathscr{H}.
\end{equation}
Let $\mathscr{C}$ be von Neumann algebra generated by $\{\Psi(v): v\in \mathscr{H}\}$. As in \cite{Gross}, for the Fock vacuum $\Omega\in\Lambda_0(\mathscr{H})\subseteq\Lambda(\mathscr{H})$,  we define
\begin{equation*}
m(f):=\langle\Omega, f\Omega\rangle_{\Lambda(\mathscr{H})},\quad f\in \mathscr{C}.
\end{equation*}
By \cite{S.}, $m(\cdot)$ is a faithful, normal, central state on $\mathscr{C}$. The space $(\mathscr{C}, m)$ is a quantum (noncommutative) probability space. 
For  $p\in[1,\infty)$, we define the noncommutative $L^p$-norm on $\mathscr{C}$ by
\begin{equation*}
  \|f\|_{p}:=m\left(|f|^p\right)^\frac{1}{p}=\left\langle\Omega, |f|^p\Omega\right\rangle_{\Lambda(\mathscr{H})}^\frac{1}{p},
\end{equation*}
where $|f|=(f^*f)^\frac{1}{2}$. The space $L^p(\mathscr{C},m)$ is the completion of $(\mathscr{C}, \|\cdot\|_p)$, which is the noncommutative $L^p$-space, abbreviated as $L^p(\mathscr{C})$.
For any von Neumann subalgebra $\mathscr{B}$  of $\mathscr{C}$,  $m(\cdot|\mathscr{B})$ denotes the conditional expectation \cite{S.,W.F} with respect to $\mathscr{B}$, that is, for any $f\in L^p(\mathscr{C})$,
\begin{equation*}
m(fg)=m(m(f|\mathscr{B})g),\quad g\in \mathscr{B}.
\end{equation*}

In this paper, let $\mathscr{H}=L^2(\mathbb R^+)$, $Jf=\bar{f}$ for $f\in L^2(\mathbb{R}^+)$, and the filtration $\{\mathscr{C}_t\}_{t\geq 0}$ of $\mathscr{C}$ be the increasing family of von Neumann subalgebras of $\mathscr{C}$ \cite{P.X} generated by 
\begin{equation*}
  \left\{\Psi(v): v\in L^2(\mathbb{R}^+)\ \textrm{and}\ \textrm{ess supp}\ v\subseteq [0,t]\right\}.
\end{equation*}
The fermion Brownian motion $W(t)$ was defined in \cite{B.S.W.1} by
\begin{equation*}
W(t):=\Psi(\chi_{[0,t]}),\quad t\geq 0.
\end{equation*}
By CAR \eqref{CAR-Psi}, it is obvious that $ W(t)^*=W(t)$ and $W(t)^2=tI.$

Inspired by the classical control theory, we have investigated the optimal control problem of the following quantum stochastic control system in fermion fields:
\begin{equation}\label{QSCS}
\left\{
\begin{aligned}
dx(t)=&\widehat{D}(t,x(t),u(t))dt+\widehat{F}(t,x(t),u(t))dW(t)+dW(t)\widehat{G}(t,x(t),u(t)),\textrm{ in } (t_0,T],\\
x(t_0)=&x_0,
\end{aligned}\right.
\end{equation} 
where the  initial condition $x_0\in L^p(\mathscr{C}_{t_0})$, the maps $\widehat{D}(\cdot,\cdot,\cdot),\widehat{F}(\cdot,\cdot,\cdot),\widehat{G}(\cdot,\cdot,\cdot):[0,T]\times L^p(\mathscr{C})\times U\to L^p(\mathscr{C})$ are adapted  for $p\geq 2$, $u(\cdot)\in U$ is control variable and $U$ is nonconvex control domain.
Physically, the solution $x(t)$ represents the observable at the time $t$ \cite{B.H.J}, such as position, momentum, angular momentum, spin, etc. And $\widehat{D}(\cdot,\cdot,\cdot)$ represents the time evolution operator, $\widehat{F}(\cdot,\cdot,\cdot), \widehat{G}(\cdot,\cdot,\cdot)$ represent the system operators, which together with the field operator $W(t)$, model the interaction of the system with the channels \cite{James.2005}. 

To the best of our knowledge, Belavkin, Gough, James et al. \cite{B.N.M,E.B,G.B.S,G.B.S-06Bellman,James.2005,James.2012} have conducted a great deal of studies on quantum stochastic control. They investigated quantum stochastic control problem by using dynamic programming principle, \textit{when the diffusion term is not controlled}. In \cite{W.W}, we have obtained the Pontryagin-type maximum principle for quantum stochastic optimal control problem of \eqref{QSCS}. The result indicates that the solution to the following second-order adjoint equation is necessary condition of quantum stochastic optimal control problem,
\begin{equation}\label{BSDE-P}
\left\{
\begin{aligned}
dP(t)=&-\{D(t)^*P(t)+P(t)D(t)+F(t)^* P(t)F(t)+F(t)^*Q(t)\Upsilon\\
&\indent+Q(t)\Upsilon F(t) -\mathbb{H}(t)\}dt+Q(t) dW(t),  \quad\textrm{in}\ [0,T),\\
P(T)=&P_T,
\end{aligned}
\right.
\end{equation}
where the terminal condition $P_T\in \mathcal{L}(L^2(\mathscr{C}_T))$, 
 $\mathbb{H}\in L^1_\mathbb{A}([0,T]; \mathcal{L}(L^2(\mathscr{C})))$, $D,F\in L^2_\mathbb{A}([0,T];\mathcal{L}(L^2(\mathscr{C})))$, and for any $t\in[0,T]$  
\begin{equation*}
D(t)=\widehat{D}_x(t,x(t),u(t)),\quad  F(t)=\widehat{F}_x(t,x(t),u(t))+\Upsilon \widehat{G}_x(t,x(t),u(t)).
\end{equation*}
Here $\Upsilon:L^2(\mathscr{C})\to L^2(\mathscr{C})$ denotes the self-adjoint unitary operator \cite{B.S.W.1,P.X} uniquely determined by 
\begin{equation*}
 \Upsilon(\Psi(v_1)\Psi(v_2)\cdots\Psi(v_n)):=\Psi(-v_1)\Psi(-v_2)\cdots\Psi(-v_n),\quad  v_i\in \mathscr{H}, \ 0\leq i\leq n.
\end{equation*}

As a continuation of  \cite{W.W}, this paper investigates the solutions to the second-order adjoint equations \eqref{BSDE-P} for quantum stochastic control system  \eqref{QSCS} in $L^2(\mathscr{C})$. 
However, by means of the result of Pisier and Xu in \cite{P.X}, the operator-valued quantum stochastic integration has no definition, even if the coefficients are multipliers in $\mathscr{C}$. In order to understand the above  operator-valued backward quantum stochastic differential equation (BQSDE, for short), we need the relaxed transposition method introduced by Lü and Zhang \cite{L.Z-2014,L.Z-2020}, which plays a crucial role in studying operator-valued backward stochastic partial differential equations.

First, we introduce the following  QSDEs driven by fermion Brownian motion:
\begin{equation}\label{QSDE-Forward-1}
 \left\{
  \begin{aligned}
  dx_1(t)&=\{D(t)x_1(t)+u_1(t)\}dt+\{F(t)x_1(t)+v_1(t)\}dW(t),\ \textrm{in}\  (t_0,T],\\
  x_1(t_0)&=\xi_1,
  \end{aligned}
  \right.
\end{equation}
and
\begin{equation}\label{QSDE-Forward-2}
  \left\{
  \begin{aligned}
  dx_2(t)&=\{D(t)x_2(t)+u_2(t)\}dt+\{F(t)x_2(t)+v_2(t)\}dW(t),\ \textrm{in}\ (t_0,T],\\
  x_2(t_0)&=\xi_2,
  \end{aligned}
  \right.
\end{equation}
where the  initial conditions $\xi_1, \xi_2\in L^2(\mathscr{C}_{t_0})$, 
$u_1, u_2,v_1,v_2\in L^2_\mathbb{A}([0,T];L^2(\mathscr{C}))$. 
By \cite[Theorem 2.1]{B.S.W.2}, the equation \eqref{QSDE-Forward-1} (\textit{resp.} \eqref{QSDE-Forward-2}) admits a unique adapted solution $x_1(\cdot)\in C_\mathbb{A}([t_0,T];L^2(\mathscr{C}))$ (\textit{resp.} $x_2(\cdot)\in C_\mathbb{A}([t_0,T];L^2(\mathscr{C}))$).


Put
\begin{align*}
\mathcal{H}_t:=&L^2(\mathscr{C}_t)\times L^2_\mathbb{A}([t,T];L^2(\mathscr{C}))\times L^2_\mathbb{A}([t,T];L^2(\mathscr{C})),\ t\in[0,T].\\
\mathcal{Q}(0,T)&:=\left\{\left(Q^{(\cdot)},\widehat{Q}^{(\cdot)}\right);\right. Q^{(t)},\widehat{Q}^{(t)}\in \mathcal{L}(\mathcal{H}_t;L^2_\mathbb{A}([t,T];L^2(\mathscr{C}))),\ \textrm{and}\\
&\indent\indent\indent\indent\left. Q^{(t)}(0,0,\cdot)^*=\widehat{Q}^{(t)}(0,0,\cdot)\ \textrm{for any}\ t\in [0,T]\right\}.
\end{align*}
\begin{defn}
The triple $\left(P(\cdot),Q^{(\cdot)},\widehat{Q}^{(\cdot)}\right)\in C_\mathbb{A}([0,T];\mathcal{L}(L^2(\mathscr{C})))\times \mathcal{Q}(0,T)$ is called to be a relaxed transposition solution to \eqref{BSDE-P} if for any $t_0\in[0,T]$, $\xi_1, \xi_2\in L^2(\mathscr{C}_{t_0})$, $u_1, u_2, v_1,v_2\in  L^2_\mathbb{A}([0,T];L^2(\mathscr{C}))$, it holds that
\begin{equation}\label{Def of relaxed transposition solution}
 \begin{aligned}
 &\langle P_Tx_2(T), x_1(T)\rangle-\int_{t_0}^T\langle \mathbb{H}(t)x_2(t),x_1(t)\rangle dt\\
=&\langle P(t_0)\xi_2, \xi_1\rangle+\int_{t_0}^T\langle P(t)u_2(t), x_1(t)\rangle dt+\int_{t_0}^T\langle P(t)x_2(t), u_1(t)\rangle dt\\
&+\int_{t_0}^T\langle P(t)F(t)x_2(t), v_1(t)\rangle dt+\int_{t_0}^T\langle P(t)v_2(t),F(t)x_1(t)+v_1(t)\rangle dt\\
&+\int_{t_0}^T\left\langle Q^{(t_0)}(\xi_2,u_2,v_2)(t), v_1(t)\right\rangle dt+\int_{t_0}^T\left\langle v_2(t), \widehat{Q}^{(t_0)}(\xi_1,u_1,v_1)(t)\right\rangle dt,
 \end{aligned}
\end{equation}
where $x_1(\cdot)$ and $x_2(\cdot)$ are the solution to \eqref{QSDE-Forward-1} and \eqref{QSDE-Forward-2}, respectively.
\end{defn}

\begin{thm}\label{Thm the relaxed transposition}
 The second-order adjoint equation \eqref{BSDE-P} admits a unique relaxed transposition solution $(P(\cdot),Q^{(\cdot)},\widehat{Q}^{(\cdot)})\in C_\mathbb{A}([0,T];\mathcal{L}(L^2(\mathscr{C})))\times \mathcal{Q}(0,T)$. Further,
\begin{equation}\label{estimate of main thm}
 \|P\|_{C_\mathbb{A}([0,T];\mathcal{L})}+\sup_{t\in [0,T]}\left\|\left(Q^{(t)},\widehat{Q}^{(t)}\right)\right\|_{\mathcal{L}(\mathcal{H}_{t};L^2([t,T];L^2(\mathscr{C})))^2}
\leq \mathcal{C}\left(\|P_T\|_{\mathcal{L}}+\|\mathbb{H}\|_{L^1_\mathbb{A}([0,T];\mathcal{L})}\right).
\end{equation}
\end{thm}


The rest of this paper is organized as follows. 
In Section \ref{Pre}, we present some preliminary results, which is needed in the sequel.   Section \ref{The main result} is devoted to proving the main result. In order to prove Theorem \ref{Thm the relaxed transposition}, we obtain the transposition solutions to the second-order adjoint equations with special date $P_T$ and $\mathbb{H}$ in Section \ref{The transpsotion solution of BQSDEs}.  We concentrate on proving Theorem \ref{Thm the relaxed transposition} in Section \ref{The realxed transpsotion solution of BQSDEs}.

\textbf{\emph{Notations.}} For any $a,b,c\in L^2(\mathscr{C})$, $\alpha,\beta\in \mathbb{C}$,
\begin{itemize}
  \item $$\langle\alpha a+\beta b, c\rangle=\bar{\alpha} \langle x, z\rangle+\bar{\beta}\langle y, z\rangle, \quad \langle  a, \alpha b+\beta c\rangle=\alpha\langle a, b\rangle+\beta\langle a,c\rangle.$$
\end{itemize}

For given $T>0$ and $p,q\in[1,\infty)$, we give the following classes processes will be used in this paper.
\begin{itemize}
 \item  $C_\mathbb{A}([0,T];L^p(\mathscr{C}))$: the space of all adapted continuous processes $f(\cdot)$ such that $f(\cdot):[0,T]\to L^p(\mathscr{C})$ with the norm 
          $ \|f\|_{C_\mathbb{A}([0,T];L^p(\mathscr{C}))}=\sup\limits_{t\in[0,T]}\|f(t)\|_p.$
 \item  $L^q_\mathbb{A}([0,T];L^p(\mathscr{C}))$ : the space of all simple adapted $L^p(\mathscr{C})$-processes $f(\cdot)$ such that $f(\cdot):[0,T]\to L^p(\mathscr{C})$  with the norm 
 $\|f\|_{L^q_\mathbb{A}([0,T];L^p(\mathscr{C}))}=\left(\int_0^T\|f(s)\|^q_{p}ds\right)^{\frac{1}{q}}.$ 
\item  $\mathcal{L}(L^2(\mathscr{C}))$: the space of all bounded linear operators on $L^2(\mathscr{C})$ with the norm  {\small$\|A\|_{\mathcal{L}}=\sup\limits_{f\in L^2(\mathscr{C}),\|f\|_2=1}\|Af\|_2$}.  
\item $\mathcal{L}_2(L^2(\mathscr{C}))$: the space of all Hilbert-Schmidt operators on $L^2(\mathscr{C})$, which is a separable Hilbert space equipped with the inner product
\begin{equation*}
\langle A, B\rangle_{\mathcal{L}_2}=\textrm{tr}(A^*B)\quad  A, B\in \mathcal{L}_2(L^2(\mathscr{C})).
\end{equation*}
And, let $\|A\|_{\mathcal{L}_2}^2:=\langle A, A\rangle_{\mathcal{L}_2}$ be called the norm of $A$ for any $A\in \mathcal{L}_2(L^2(\mathscr{C}))$. 
\end{itemize}

Here and in what follows, when there is no confusion, denote $\langle\cdot,\cdot \rangle$ and  $\langle\cdot,\cdot \rangle_{\mathcal{L}_2}$  the inner products of Hilbert spaces $L^2(\mathscr{C})$ and $\mathcal{L}_2(L^2(\mathscr{C}))$, respectively.
\section{Preliminaries}\label{Pre}
\indent\indent
 This section is dedicated to the auxiliary results.
Like in Clifford stochastic integral in $L^2(\mathscr{C})$, we introduce simple adapted processes of $\mathcal{L}_2(L^2(\mathscr{C}))$.
\begin{defn}\label{the defnition of  simple processes of mathcal(L)_2(L^2(mathscr(C)))}
A simple adapted $\mathcal{L}_2(L^2(\mathscr{C}))$-process on $[0,T]$ is a map $\phi:[0,T]\to \mathcal{L}_2(L^2(\mathscr{C}))$ such that $\phi(t)\in \mathcal{L}_2(L^2(\mathscr{C}_t))$ for all $t\in[0,T]$, and 
\begin{equation*}
 \phi=\sum_{k\geq 0}\phi(t_k)\chi_{[t_k,t_{k+1})},
\end{equation*}
where $\{t_k\}_{k=0}^n$ is a partition of $[0,T]$ and $\phi(t_k)\in \mathcal{L}_2(L^2(\mathscr{C}_{t_k}))$, $0\leq k\leq n-1$.
\end{defn}

If $\phi =\sum_{k} \phi(t_k)\chi_{[t_k,t_{k+1})}$ is a simple adapted $\mathcal{L}_2(L^2(\mathscr{C}))$-process on $[0,T]$, the stochastic integral of $\phi$ over $[0,T]$ is
\begin{equation*}
  \int_0^T\phi(s)dW(s)=\sum_{k=0}^{n-1}\phi(t_k)(W(t_{k+1})-W(t_k)).
\end{equation*}
Clearly, $\int_{0}^{T}\phi(s)dW(s)\in \mathcal{L}_2(L^2(\mathscr{C}_T))$, and is independent of decomposition of $\phi$.

\begin{lem}\label{ito isomerty of stochastic integral valued in H-S}
If $\phi =\sum_k \phi(t_k)\chi_{[t_k,t_{k+1})}$ is a simple adapted $\mathcal{L}_2(L^2(\mathscr{C}))$-process on $[0,T]$, 
 then 
\begin{equation*}
\left\|\int_0^T\phi(s)dW(s)\right\|^2_{\mathcal{L}_2}=\int_0^T\|\phi(s)\|^2_{\mathcal{L}_2}ds. 
\end{equation*}
\end{lem}
\begin{proof} Since $\phi$ is a simple adapted $\mathcal{L}_2(L^2(\mathscr{C}))$-process on $[0,T]$, we can obtain that
\begin{equation*}
\begin{aligned}
&\left\|\int_0^T\phi(s)dW(s)\right\|^2_{\mathcal{L}_2}\\
=&\textrm{tr}\left(\left(\int_0^T\phi(s)dW(s)\right)^*\left(\int_0^T\phi(s)dW(s)\right)\right)\\
=&\sum_{i,j=0}^{n-1}\sum_{k=1}^{\infty}\langle e_k,\ \{\phi(t_i)(W(t_{i+1})-W(t_i))\}^*\{\phi(t_j)(W(t_{j+1})-W(t_j))\}e_k\rangle\\
=&\sum_{i,j=0}^{n-1}\sum_{k=1}^{\infty}\langle\phi(t_i)(W(t_{i+1})-W(t_i)) e_k,\ \phi(t_j)(W(t_{j+1})-W(t_j)) e_k\rangle\\
=&\sum_{i,j=0}^{n-1}\sum_{k=1}^{\infty}\langle\phi(t_i)\Upsilon e_k (W(t_{i+1})-W(t_i)),\ \phi(t_j)\Upsilon e_k (W(t_{j+1})-W(t_j))\rangle\\
=&\sum_{i=0}^{n-1}\textrm{tr}(\phi(t_i)^*\phi(t_i))(t_{i+1}-t_i)\\
=&\int_0^T\|\phi(t)\|_{\mathcal{L}_2}^2dt,
\end{aligned}
\end{equation*}
where  $\{e_j\}_{j=1}^\infty$ is an orthonormal basis of $L^2(\mathscr{C})$, and the fourth equality and the fifth equality come from \cite[Lemma 3.15 and Theorem 3.5(c)]{B.S.W.1}, respectively.
\end{proof} 

\begin{defn}
For $p\in[1,\infty)$, let $L^p_\mathbb{A}([0,T];\mathcal{L}_2(L^2(\mathscr{C})))$ be the completion of all simple adapted $\mathcal{L}_2(L^2(\mathscr{C}))$-processes with the norm 
$\|\phi\|_{L^p_\mathbb{A}([0,T];\mathcal{L}_2)}:=\left(\int_{0}^{T}\|\phi(t)\|^p_{\mathcal{L}_2}dt\right)^\frac{1}{p}.$
 $L^\infty_\mathbb{A}([0,T];\mathcal{L}_2(L^2(\mathscr{C})))$ is the space of all adapted measurable processes $\phi(\cdot)$ such that $\phi(\cdot):[0,T]\to\mathcal{L}_2(L^2(\mathscr{C}))$ with the norm  $\|\phi\|_{L^\infty_\mathbb{A}([0,T];\mathcal{L}_2)}:={\rm{ess}} \sup_{t\in[0,T]}\|\phi(t)\|_{\mathcal{L}_2}$. 
 $C_\mathbb{A}([0,T];$ $\mathcal{L}_2(L^2(\mathscr{C})))$ is the space of all adapted processes $\phi(\cdot)$ such that {\small$\phi(\cdot):[0,T]\to\mathcal{L}_2(L^2(\mathscr{C}))$} is continuous, with the norm $\|\phi\|_{C_\mathbb{A}([0,T];\mathcal{L}_2)}:=\sup\limits_{t\in[0,T]}\|\phi(t)\|_{\mathcal{L}_2}.$
\end{defn}

 By Lemma \ref{ito isomerty of stochastic integral valued in H-S},
 the  operator-valued stochastic integral $\int_0^T\Phi(t)dW(t)$ is well-defined for any $\Phi(\cdot)\in L_\mathbb{A}^2([0,T];\mathcal{L}_2(L^2(\mathscr{C})))$, and
\begin{equation}\label{isomertry of H-S operator}
\left\|\int_0^T\Phi(t)dW(t)\right\|_{\mathcal{L}_2}^2=\int_{0}^{T}\|\Phi(t)\|_{\mathcal{L}_2}^2dt.
\end{equation}

In a similar way,  a simple adapted $\mathcal{L}(L^2(\mathscr{C}))$-process on $[0,T]$ is a map $\phi:[0,T]\to \mathcal{L}(L^2(\mathscr{C}))$ such that $\phi(t)\in \mathcal{L}(L^2(\mathscr{C}_t))$ for all $t\in[0,T]$, and  $\phi=\sum_{k\geq 0}\phi(t_k)\chi_{[t_k,t_{k+1})}$,
where $\{t_k\}_{k=0}^n$ is a partition of $[0,T]$ and $\phi(t_k)\in \mathcal{L}(L^2(\mathscr{C}_{t_k}))$, $0\leq k\leq n-1$. The space of simple adapted $\mathcal{L}(L^2(\mathscr{C}))$-processes is a normed space with the norm {\small$\|\phi\|_{L^p_\mathbb{A}([0,T];\mathcal{L})}:=\left(\int_{0}^{T}\|\phi(t)\|^p_{\mathcal{L}}dt\right)^\frac{1}{p}$}, and let $L^p_\mathbb{A}([0,T];\mathcal{L}(L^2(\mathscr{C})))$ be its completion. Moreover, let
\begin{equation*}
\begin{aligned}
L^\infty_\mathbb{A}([0,T];\mathcal{L}(L^2(\mathscr{C}))):=&\Big\{\phi:[0,T]\to \mathcal{L}(L^2(\mathscr{C})); \phi(\cdot) \textrm{ is  adapted, measurable }\\
&\indent\textrm{and }\|\phi\|_{L^\infty_\mathbb{A}([0,T];\mathcal{L})}:={{\rm{ess}} \sup}_{t\in[0,T]}
\|\phi(t)\|_{\mathcal{L}}<\infty\Big\},\\
C_\mathbb{A}([0,T];\mathcal{L}(L^2(\mathscr{C})))
:=&\Big\{\phi:[0,T]\to \mathcal{L}(L^2(\mathscr{C})); \phi(\cdot) \textrm{ is  adapted, continuous}\\
&\indent\indent\textrm{ and }\|\phi\|_{C_\mathbb{A}([0,T];\mathcal{L})}:=\sup_{t\in[0,T]}\|\phi(t)\|_{\mathcal{L}}<\infty\Big\}.
\end{aligned}
\end{equation*}

\begin{thm}\label{Riesz representation}
For  $p\in[1,\infty)$ with $\frac{1}{p}+\frac{1}{q}=1$, it holds that
\begin{equation}\label{dual relation}
\left(L^p_\mathbb{A}([0,T];\mathcal{L}_2(L^2(\mathscr{C})))\right)^*= L^q_\mathbb{A}([0,T];\mathcal{L}_2(L^2(\mathscr{C}))).
\end{equation}
\end{thm}
\begin{proof} 
For any $t\in[0,T]$, let $\mathcal{P}_t:\mathcal{L}_2(L^2(\mathscr{C}))\to \mathcal{L}_2(L^2(\mathscr{C}_t))$ be the orthogonal projection. 
Let 
\begin{equation*}
 \mathfrak{S}:=\left\{\sum_{k}a_k\chi_{[t_k,t_{k+1})}; a_k\in \mathcal{L}_2(L^2(\mathscr{C}))\ \textrm{and}\ \{t_k\}_{k=0}^n \textrm{ is a partition of } [0,T]\right\},
\end{equation*}
which is dense in $L^p([0,T];\mathcal{L}_2(L^2(\mathscr{C})))$. Denote $\mathbb{E}:\mathfrak{S} \to L_\mathbb{A}^p([0,T];\mathcal{L}_2(L^2(\mathscr{C})))$ by
\begin{equation*}
(\mathbb{E}h)(t)=\sum_{k}(\mathcal{P}_{t_k}a_k)\chi_{[t_k,t_{k+1})}(t),\  t\in[0,T],
\end{equation*}
then $\|\mathbb{E}\|\leq 1$. Therefore, $\mathbb{E}$ can be extended to 
\begin{equation*}
\widehat{\mathbb{E}}:L^p([0,T];\mathcal{L}_2(L^2(\mathscr{C})))\to L_\mathbb{A}^p([0,T];\mathcal{L}_2(L^2(\mathscr{C})))
\end{equation*}
 with $\left\|\widehat{\mathbb{E}}\right\|\leq 1$. For any $h\in L^p([0,T];\mathcal{L}_2(L^2(\mathscr{C})))$, there exists a sequence $\{
h_n\}_{n\geq0}\subseteq \mathfrak{S}$ such that 
\begin{equation*}
\lim\limits_{n\to\infty}\|h_n-h\|_{L^p([0,T];\mathcal{L}_2)}=0.
\end{equation*}
Then there exists a subsequence $\{h_{n_k}\}_{n_k\geq 0}\subseteq\{
h_n\}_{n\geq0} $ such that $h_{n_k}\to h$, a.e.    
Obviously, $\mathbb{E}h_{n_k}\to \widehat{\mathbb{E}}h\in L_\mathbb{A}^1([0,T];\mathcal{L}_2(L^2(\mathscr{C})))$.
Besides, 
\begin{equation*}
\left(\mathbb{E}h_{n_k}\right)(t)=\mathcal{P}_th_{n_k}(t)\to \mathcal{P}_th(t),\ \textrm{a.e.}\ t\in[0,T].
\end{equation*}
 Therefore, 
\begin{equation*}
   (\widehat{\mathbb{E}}h)(t)=\mathcal{P}_th(t),\ \textrm{a.e.}\ t\in[0,T].
 \end{equation*}
 
For any $\ell\in \left(L^p_\mathbb{A}([0,T];\mathcal{L}_2(L^2(\mathscr{C})))\right)^*$, 
by Hahn-Banach Theorem, there exists $\widehat{\ell}\in \left(L^p([0,T];\mathcal{L}_2(L^2(\mathscr{C})))\right)^*$ such that
\begin{equation*}
\ell=\widehat{\ell}|_{L^p_\mathbb{A}([0,T];\mathcal{L}_2(L^2(\mathscr{C})))}, \quad \|\ell\|_{\left(L^p_\mathbb{A}([0,T];\mathcal{L}_2(L^2(\mathscr{C})))\right)^*}=\left\|\widehat{\ell}\right\|_{\left(L^p([0,T];\mathcal{L}_2(L^2(\mathscr{C})))\right)^*}.
\end{equation*}
In view of 
\begin{equation*}
 \left( L^p([0,T];\mathcal{L}_2(L^2(\mathscr{C})))\right)^*=  L^q([0,T];\mathcal{L}_2(L^2(\mathscr{C}))),
\end{equation*}
thus, there exists $g\in L^q([0,T];\mathcal{L}_2(L^2(\mathscr{C})))$ such that
\begin{equation*}
\widehat{\ell}(f)=\int_0^T\langle g(t), f(t)\rangle_{\mathcal{L}_2}dt,\quad  f\in L^p([0,T];\mathcal{L}_2(L^2(\mathscr{C}))),
\end{equation*} 
and 
\begin{equation*}
\left\|\widehat{\ell}\right\|_{ ( L^p([0,T];\mathcal{L}_2(L^2(\mathscr{C}))))^*}=\|g\|_{L^q([0,T];\mathcal{L}_2)}.
\end{equation*}
 Then for any  $f\in L^p_\mathbb{A}([0,T];\mathcal{L}_2(L^2(\mathscr{C})))$, 
\begin{equation}\label{widehat(ell)}
\ell(f)= \widehat{\ell}(f)=\int_0^T\langle g(t),f(t)\rangle_{\mathcal{L}_2}dt=\int_0^T\left\langle\left(\widehat{\mathbb{E}}g\right)(t), f(t)\right\rangle_{\mathcal{L}_2}dt.
\end{equation}
For  $g\in L^q([0,T];\mathcal{L}_2(L^2(\mathscr{C})))\subseteq L^1([0,T];\mathcal{L}_2(L^2(\mathscr{C})))$, 
$\widehat{\mathbb{E}}g\in L^1_\mathbb{A}([0,T];\mathcal{L}_2(L^2(\mathscr{C})))$ is measurable. Further, 
$\left\|\widehat{\mathbb{E}}g\right\|_{L^q([0,T];\mathcal{L}_2)}\leq \|g\|_{L^q([0,T];\mathcal{L}_2)}$ and $\left(\widehat{\mathbb{E}}g\right)(t)\in\mathcal{L}_2(L^2(\mathscr{C}_t))$ a.e. $t\in[0,T]$, thus,  
$\widehat{\mathbb{E}}g\in L^q_\mathbb{A}([0,T];\mathcal{L}_2(L^2(\mathscr{C})))$.

By \eqref{widehat(ell)} and H\"{o}lder inequality, it is obvious that 
\begin{equation*}
\|\ell\|_{\left(L^p_\mathbb{A}([0,T];\mathcal{L}_2(L^2(\mathscr{C})))\right)^*}\leq \left\|\widehat{\mathbb{E}}g\right\|_{L^q_\mathbb{A}([0,T];\mathcal{L}_2)}.
\end{equation*}
On the other hand,
\begin{equation*}
\|\ell\|_{\left(L^p([0,T];\mathcal{L}_2(L^2(\mathscr{C})))\right)^*}=\left\|\widehat{\ell}\right\|_{\left(L^p_\mathbb{A}([0,T];\mathcal{L}_2(L^2(\mathscr{C})))\right)^*}=\|g\|_{L^q_\mathbb{A}([0,T];\mathcal{L}_2)}\geq\left\|\widehat{\mathbb{E}}g\right\|_{L^q_\mathbb{A}([0,T];\mathcal{L}_2)}.
\end{equation*}
Then, $\|\ell\|_{\left(L^p_\mathbb{A}([0,T];\mathcal{L}_2(L^2(\mathscr{C})))\right)^*}=\left\|\widehat{\mathbb{E}}g\right\|_{L^q_\mathbb{A}([0,T];\mathcal{L}_2)}$. Hence, \eqref{dual relation} holds.
\end{proof}

Combined with the It\^{o} isometry of Clifford martingales \cite{B.S.W.1,B.S.W.2,B.S.W.3,P.X} in $L^2(\mathscr{C})$ and CAR \eqref{CAR-Psi}, we have the following results about the solution to \eqref{QSDE-Forward-1}.
\begin{lem}\label{the relationship between solution and coefficients}
For any $D,F\in L^2_\mathbb{A}([0,T];\mathcal{L}(L^2(\mathscr{C})))$, let
\begin{equation}\label{the definition of M_(D,F)}
  M_{D,F,2}(\cdot):=\|D(\cdot)\|^2_{\mathcal{L}}+\|F(\cdot)\|^2_{\mathcal{L}}.
\end{equation}
Then for any $t_0\in[0,T]$, the following conclusions hold:\\
\textrm{(1)} If $u_1=v_1=0$ in \eqref{QSDE-Forward-1}, then there exists an operator 
\begin{equation*}
U(\cdot,t_0)\in \mathcal{L}(L^2(\mathscr{C}_{t_0});C_\mathbb{A}([t_0,T];L^2(\mathscr{C})))
\end{equation*}
such that the solution to \eqref{QSDE-Forward-1} can be represented as $x_1(\cdot)=U(\cdot,t_0)\xi_1$. Further, for any $t\in[0,T)$, $\xi_1\in L^2(\mathscr{C}_{t_0})$ and $\varepsilon>0$, there is a $\delta \in[0,T-t_0)$ such that for any $t\in[t_0,t_0+\delta]\subseteq [t_0,T]$, it holds that
\begin{equation}\label{u_1=v_1=0}
  \|U(\cdot,t_0)\xi_1-U(\cdot,t)\xi_1\|_{L^\infty([0,T];L^2(\mathscr{C}))}<\varepsilon.
 \end{equation}
\textrm{(2)} If $\xi_1=0$ and $v_1=0$ in \eqref{QSDE-Forward-1}, then there exists an operator 
\begin{equation*}
V(\cdot,t_0)\in \mathcal{L}(L^2_\mathbb{A}([t_0,T]; L^2(\mathscr{C}));C_\mathbb{A}([t_0,T];L^2(\mathscr{C})))
\end{equation*}
 such that the solution to \eqref{QSDE-Forward-1} can be represented as $x_1(\cdot)=V(\cdot,t_0)u_1$.\\
\textrm{(3)}If $\xi_1=0$ and $u_1=0$ in \eqref{QSDE-Forward-1}, then there exists an operator 
\begin{equation*}
\Xi(\cdot,t_0)\in \mathcal{L}(L^2_\mathbb{A}([t_0,T];L^2(\mathscr{C}));C_\mathbb{A}([t_0,T];L^2(\mathscr{C})))
\end{equation*}
 such that the solution to \eqref{QSDE-Forward-1} can be represented as $x_1(\cdot)=\Xi(\cdot,t_0)v_1$.
\end{lem} 
\begin{proof} The proof is similar to that of \cite[Lemma 12.10]{L.Z-2020}. For the reader's convenience, we only give the proof of $(1)$.

Define $U(\cdot,t_0):L^2(\mathscr{C}_{t_0})\rightarrow C_\mathbb{A}([t_0,T];L^2(\mathscr{C}))$ by
\begin{equation*}
  U(t, t_0)\xi_1=x_1(t),\quad t\in[t_0,T],
\end{equation*}
where $x_1(t)=U(t,t_0)\xi_1$ is the solution to \eqref{QSDE-Forward-1} with $u_1=v_1=0$.
It follows from \cite[Theorem 2.1]{B.S.W.2} that
\begin{equation*}
 \|x_1(\cdot)\|_{C_\mathbb{A}([t_0,T];L^2(\mathscr{C}))}^2\leq \mathcal{C}\|\xi_1\|_2^2.
\end{equation*}
This shows that $U(\cdot,t_0)$ is bounded operator from $L^2(\mathscr{C}_{t_0})$ to $C_\mathbb{A}([t_0,T];L^2(\mathscr{C}))$. 

Furthermore, by the definition of $U(\cdot,t_0)$ and $U(\cdot,t)$, for any $t_0\leq t\leq r\leq T$, we have that 
\begin{equation*}
U(r,t)\xi_1
=\xi_1+\int_t^rD(s)U(s,t)\xi_1ds+\int_t^rF(s)U(s,t)\xi_1dW(s),
\end{equation*}
and
\begin{equation*}
U(r,t_0)\xi_1
=\xi_1+\int_{t_0}^rD(s)U(s,t_0)\xi_1ds+\int_{t_0}^rF(s)U(s,t_0)\xi_1dW(s).
\end{equation*}
Hence,
\begin{equation}\label{the sum of U(r,t)-U(r,t_0)}
\begin{aligned}
&\|\{U(r,t)-U(r,t_0)\}\xi_1\|_2\\
\leq&\left\|\int_t^r D(s)\{U(s,t)-U(s,t_0)\}\xi_1ds\right\|_2+\left\|\int_{t_0}^tD(s)U(s,t_0)\xi_1ds\right\|_2\\
&+\left\|\int_t^rF(s)\{U(s,t)-U(s,t_0)\}\xi_1dW(s)\right\|_2+\left\|\int_{t_0}^tF(s)U(s,t_0)\xi_1dW(s)\right\|_2.
\end{aligned}
\end{equation}
By the H\"{o}lder inequality, we obtain that
\begin{equation}\label{drift-plus}
\begin{aligned}
\left\|\int_t^rD(s)\{U(s,t)-U(s,t_0)\}\xi_1ds\right\|_2^2
\leq&\left(\int_t^r\left\|D(s)\{U(s,t)-U(s,t_0)\}\xi_1\right\|_2 ds\right)^2\\
\leq& \mathcal{C}\int_t^r\left\|D(s)\right\|^2_{\mathcal{L}}\left\|\{U(s,t)-U(s,t_0)\}\xi_1\right\|_2^2 ds,
\end{aligned}
\end{equation}
and
\begin{equation}\label{drift}
\left\|\int_{t_0}^tD(s)U(s,t_0)\xi_1ds\right\|_2^2 \leq\mathcal{C}\int_{t_0}^t\|D(s)\|^2_{\mathcal{L}}\left\|U(s,t_0)\xi_1\right\|_2^2 ds.
\end{equation}
By using the It\^{o} isometry  of Clifford martingales \cite[Theorem 3.5]{B.S.W.1}, we have
\begin{equation}\label{diffsion-plus}
\begin{aligned}
\left\|\int_t^rF(s)\{U(s,t)-U(s,t_0)\}\xi_1dW(s)\right\|_2^2
=&\int_t^r\left\|F(s)\{U(s,t)-U(s,t_0)\}\xi_1\right\|_2^2 ds\\
\leq&\int_t^r\left\|F(s)\right\|^2_{\mathcal{L}}\left\|\{U(s,t)-U(s,t_0)\}\xi_1\right\|_2^2 ds.
\end{aligned}
\end{equation}
Similarly,
\begin{equation}\label{diffsion}
 \left\|\int_{t_0}^tF(s)U(s,t_0)\xi_1dW(s)\right\|_2^2
 \leq \int_{t_0}^t\|F(s)\|^2_{\mathcal{L}}\left\|U(s,t_0)\xi_1\right\|_2^2 ds.
\end{equation}

From \eqref{the definition of M_(D,F)}, \eqref{the sum of U(r,t)-U(r,t_0)}-\eqref{diffsion}, we obtain that
\begin{equation*}
\begin{aligned}
\|\{U(r,t)-U(r,t_0)\}\xi_1\|_2^2\leq &\mathcal{C}\int_t^rM_{D,F,2}(s)\left\|\{U(s,t)-U(s,t_0)\}\xi_1\right\|_2^2 ds\\
&+\mathcal{C}\int_{t_0}^tM_{D,F,2}(s)\left\|U(s,t_0)\xi_1\right\|_2^2 ds.
\end{aligned}
\end{equation*}
Then, by the Gronwall inequality, we find that
\begin{equation*}
\|\{U(r,t)-U(r,t_0)\}\xi_1\|_2^2\leq \mathcal{C}e^{\int_{t_0}^tM_{D,F,2}(\tau)d\tau}\int_{t}^rM_{D,F,2}(s)ds. 
\end{equation*}
Since $M_{D,F,2}(\cdot)\in L^1([0,T])$, $e^{\int_{t_0}^tM_{D,F,2}(\tau)d\tau}\leq \mathcal{C}e^{t-t_0}$,  
%
we can conclude that
\begin{equation*}
 \lim_{t\to t_0+}\|\{U(r,t)-U(r,t_0)\}\xi_1\|_2^2\leq\mathcal{C}\lim_{t\to t_0+}e^{\int_{t_0}^tM_{D,F,2}(\tau)d\tau}\int_{t}^rM_{D,F,2}(s)ds= 0.
\end{equation*}
Hence, there is $\delta\in(0,T-t_0)$ such that \eqref{u_1=v_1=0} holds for any $t\in[t_0,t_0+\delta]$.
\end{proof}

\section{The Solution to BQSDEs}\label{The main result}
\indent\indent
In this section, we prove that the second-order adjoint  equation \eqref{BSDE-P} admits a unique relaxed transposition solution. 
To make it, we introduce the transposition solution with  $P_T\in \mathcal{L}_2(L^2(\mathscr{C}_T))$, $\mathbb{H}\in L_\mathbb{A}^1([0,T];\mathcal{L}_2(L^2(\mathscr{C})))$.

\begin{defn}\label{The transposition solution}
The pair 
\begin{equation*}
  (P(\cdot),Q(\cdot))\in C_\mathbb{A}([0,T];\mathcal{L}_2(L^2(\mathscr{C})))\times L_\mathbb{A}^2\left([0,T];\mathcal{L}_2(L^2(\mathscr{C}))\right)
\end{equation*}
 is called to be a transposition solution to \eqref{BSDE-P} if for any $t_0\in[0,T]$, $\xi_1,\xi_2\in L^2(\mathscr{C}_{t_0})$, $u_1, u_2, v_1,v_2\in L_\mathbb{A}^2([0,T];L^2(\mathscr{C}))$, it holds that
\begin{equation}\label{defn-the transposition solution}  
\begin{aligned}
 &\langle P_Tx_2(T), x_1(T)\rangle-\int_{t_0}^T\langle \mathbb{H}(t)x_2(t),x_1(t)\rangle dt\\
=&\langle P(t_0)\xi_2, \xi_1\rangle+\int_{t_0}^T\langle P(t)u_2(t), x_1(t)\rangle dt+\int_{t_0}^T\langle P(t)x_2(t), u_1(t)\rangle dt\\
&+\int_{t_0}^T\langle P(t)F(t)x_2(t), v_1(t)\rangle dt+\int_{t_0}^T\langle P(t)v_2(t),F(t)x_1(t)+v_1(t)\rangle dt  \\
& +\int_{t_0}^T\langle Q(t)\Upsilon v_2(t), x_1(t)\rangle dt+\int_{t_0}^T\langle Q(t)\Upsilon x_2(t), v_1(t)\rangle dt,
\end{aligned}
\end{equation}
where $x_1(\cdot)$ and $x_2(\cdot)$ solve \eqref{QSDE-Forward-1} and \eqref{QSDE-Forward-2}, respectively.
\end{defn}

\subsection{The transposition solution to BQSDEs}\label{The transpsotion solution of BQSDEs}
\indent\indent
This subsection is dedicated to the transposition solution to the second-order adjoint equation \eqref{BSDE-P}. 
\begin{thm}\label{Thm-the transposition solution}
Assume that $P_T\in \mathcal{L}_2(L^2(\mathscr{C}_T))$, $\mathbb{H}\in L_\mathbb{A}^\infty([0,T];\mathcal{L}_2(L^2(\mathscr{C})))$. Then the equation \eqref{BSDE-P} admits a unique transposition solution $(P(\cdot),Q(\cdot))$. 
Further,
\begin{equation}\label{estimate of the transposition solution}
  \|(P,Q)\|_{C_\mathbb{A}([0,T];\mathcal{L}_2)\times L_\mathbb{A}^2([0,T];\mathcal{L}_2) }\leq \mathcal{C}\left(\|P_T\|_{\mathcal{L}_2}+\|\mathbb{H}\|_{L^1_\mathbb{A}([0,T];\mathcal{L}_2)}\right).
\end{equation}
\end{thm}
\begin{proof}
Because the proof is too long, we divide the proof into several steps.

\textbf{Step 1.} We shall prove that the following linear BQSDE admits a unique transposition solution 
\begin{equation}\label{BQSDE-P-linear}
\left\{
 \begin{aligned}
&dP=g(t)dt+Q(t)dW(t),\quad \textrm{in}\ [0,T),\\
&P(T)=P_T,
 \end{aligned}
\right.
\end{equation}
where the terminal condition $P_T\in \mathcal{L}_2(L^2(\mathscr{C}_T))$, $g\in L_\mathbb{A}^1([0,T];\mathcal{L}_2(L^2(\mathscr{C})))$. 
To get the desired result, we introduce the following QSDE:
\begin{equation}\label{QSDE-X}
\left\{
 \begin{aligned}
&dX=z(t)dt+w(t)dW(t),\quad \textrm{in}\ (t_0,T],\\
&X(t_0)=X_0,
 \end{aligned}
\right.
\end{equation}
where the initial condition $X_0\in \mathcal{L}_2(L^2(\mathscr{C}_{t_0}))$,  $z\in L^1_\mathbb{A}([0,T];\mathcal{L}_2(L^2(\mathscr{C})))$ and $w\in L^2_\mathbb{A}([0,T];\mathcal{L}_2(L^2(\mathscr{C}))$. By \eqref{isomertry of H-S operator} and Picard iteration, similar to \cite[Theorem 2.1]{B.S.W.2}, we prove that $X(\cdot)\in C_\mathbb{A}([0,T];\mathcal{L}_2(L^2(\mathscr{C})))$ is the unique solution to  \eqref{QSDE-X}, and 
\begin{equation}\label{estimate of X}
\begin{aligned}
&\sup_{t\in[t_0,T]}\|X(t)\|_{\mathcal{L}_2} \leq\mathcal{C}\left(\|X_0\|_{\mathcal{L}_2}+\|z\|_{L^1_\mathbb{A}([0,T];\mathcal{L}_2)}+\|w\|_{L^2_\mathbb{A}([0,T];\mathcal{L}_2)}\right).
\end{aligned}
\end{equation}

For any $t_0\in [0,T]$, we define a linear functional $l$ on $L^1_\mathbb{A}([t_0,T];\mathcal{L}_2(L^2(\mathscr{C})))\times L^2_\mathbb{A}([t_0,T];\mathcal{L}_2(L^2(\mathscr{C})))\times \mathcal{L}_2(L^2(\mathscr{C}))$ as follows:
\begin{equation}\label{the definition of linear functional}
l(z(\cdot),w(\cdot),X_0):=\langle P_T,X(T)\rangle_{\mathcal{L}_2}-\int_{t_0}^{T}\langle g(t),X(t)\rangle_{\mathcal{L}_2}dt.
\end{equation}
By H\"{o}lder inequality and \eqref{estimate of X}, we find that
\begin{equation}\label{estimate of X and P}
\begin{aligned}
|l(z(\cdot),w(\cdot),X_0)|
\leq&\|P_T\|_{\mathcal{L}_2}\|X(T)\|_{\mathcal{L}_2}+\|g\|_{L^1_\mathbb{A}([t_0,T];\mathcal{L}_2)}\sup_{t\in[t_0,T]}\|X(t)\|_{\mathcal{L}_2}\\
\leq& \mathcal{C}(\|P_T\|_{\mathcal{L}_2}+\|g\|_{L^1_\mathbb{A}([t_0,T];\mathcal{L}_2)})\\
& \cdot\|(z(\cdot),w(\cdot),X_0)\|_{L^1_\mathbb{A}([t_0,T];\mathcal{L}_2)\times L^2_\mathbb{A}([t_0,T];\mathcal{L}_2)\times \mathcal{L}_2}, \quad
t_0\in[0,T],
\end{aligned}
\end{equation}
where the positive constant $\mathcal{C}$ is independent of $t_0$. 
Thus, $l$ is a bounded linear functional on $L^1_\mathbb{A}([t_0,T];\mathcal{L}_2(L^2(\mathscr{C})))\times L_\mathbb{A}^2([t_0,T];\mathcal{L}_2(L^2(\mathscr{C})))\times \mathcal{L}_2(L^2(\mathscr{C}))$. By  Theorem \ref{Riesz representation}, we can find $P^{t_0}(\cdot)\in L_\mathbb{A}^\infty([t_0,T];\mathcal{L}_2(L^2(\mathscr{C})))$,  $Q^{t_0}(\cdot)\in L_\mathbb{A}^2([t_0,T];\mathcal{L}_2(L^2(\mathscr{C})))$, and $\zeta^{t_0}\in \mathcal{L}_2(L^2(\mathscr{C}_{t_0}))$ such that
\begin{equation}\label{ito formula-X-P}
\begin{aligned}
& \langle P_T, X(T)\rangle_{\mathcal{L}_2}-\int_{t_0}^T\langle g(t) ,X(t)\rangle_{\mathcal{L}_2} dt\\
=&\int_{t_0}^T\langle P^{t_0}(t),z(t)\rangle_{\mathcal{L}_2}dt +\int_{t_0}^T\langle Q^{t_0}(t), w(t) \rangle_{\mathcal{L}_2} dt+\langle\zeta^{t_0},X_0\rangle_{\mathcal{L}_2}.
 \end{aligned}
\end{equation}
Clearly, $\zeta^T=P_T$. Further, there is a positive constant $\mathcal{C}$ such that, for any $t_0\in[0,T]$,
\begin{equation*}
\begin{aligned}
\|\left(P^{t_0}(\cdot),Q^{t_0}(\cdot), \zeta^{t_0}\right)\|_{L_\mathbb{A}^\infty([t_0,T];\mathcal{L}_2)\times L_\mathbb{A}^2([t_0,T];\mathcal{L}_2)\times \mathcal{L}_2}
\leq\mathcal{C}\left(\|g\|_{L^1_\mathbb{A}([t_0,T];\mathcal{L}_2)}+\|P_T\|_{\mathcal{L}_2} \right).
\end{aligned}
\end{equation*}

Next, we shall show the time consistency of $(P^{t_0}(\cdot),Q^{t_0}(\cdot))$, that is, for any $t_1$ and $t_2$ satisfying $0\leq t_1\leq t_2\leq T$, it holds that
\begin{equation}\label{the time consistency of t1 and t2}
 \left(P^{t_1}(t),Q^{t_1}(t)\right)=\left(P^{t_2}(t),Q^{t_2}(t)\right),\quad t\in[t_2,T].
\end{equation}
Let $X^{t_0}(\cdot)$ be the solution to \eqref{QSDE-X} with the initial time $t_0$. To show \eqref{the time consistency of t1 and t2}, we fix $\rho\in L^1_\mathbb{A}([t_0,T];\mathcal{L}_2(L^2(\mathscr{C})))$ and $\psi\in L^2_\mathbb{A}([t_0,T];\mathcal{L}_2(L^2(\mathscr{C})))$, and choose $t_0=t_2$, $X_0=0$, $z(\cdot)=\rho(\cdot)$ and $w(\cdot)=\psi(\cdot)$ in \eqref{QSDE-X}. From \eqref{ito formula-X-P}, we obtain that
\begin{equation}\label{ito formula-X-P-1}
\begin{aligned}
&\left\langle P_T, X^{t_2}(T)\right\rangle_{\mathcal{L}_2}-\int_{t_2}^T\left\langle g(t), X^{t_2}(t) \right\rangle_{\mathcal{L}_2}dt 
 =\int_{t_2}^T\left\langle P^{t_1}(t), \rho(t)\right\rangle_{\mathcal{L}_2}+\left\langle Q^{t_1}(t),\psi(t) \right\rangle_{\mathcal{L}_2} dt.
 \end{aligned}
\end{equation}
Then, we choose $t_0=t_1$, $X_0=0$, $z(t)=\chi_{[t_2,T]}(t)\rho(t)$ and $w(t)=\chi_{[t_2,T]}(t)\psi(t)$ in \eqref{QSDE-X}. Clearly, 
$X^{t_1}(t)=\left\{\begin{aligned}&X^{t_2}(t),\quad& t\in[t_2,T],\\
&0,\quad& t\in[t_1,t_2].\end{aligned}
\right.$
From \eqref{ito formula-X-P-1}, we have that
\begin{equation}\label{ito formula-X-P-2}
\begin{aligned}
&\left\langle P_T, X^{t_2}(T)\right\rangle_{\mathcal{L}_2}-\int_{t_2}^T\left\langle g(t), X^{t_2}(t)\right\rangle_{\mathcal{L}_2} dt
=\int_{t_2}^T\left\langle P^{t_2}(t), \rho(t) \right\rangle_{\mathcal{L}_2} +\left\langle Q^{t_2}(t), \psi(t)\right\rangle_{\mathcal{L}_2} dt.
 \end{aligned}
\end{equation}
From \eqref{ito formula-X-P-1} and \eqref{ito formula-X-P-2}, we obtain that,
for any $\rho\in L^1_\mathbb{A}([t_0,T];\mathcal{L}_2(L^2(\mathscr{C})))$ and $\psi\in L^2_\mathbb{A}([t_0,T];\mathcal{L}_2(L^2(\mathscr{C})))$,
\begin{equation}\label{ito-X-P}
 \int_{t_2}^T\left\langle P^{t_1}(t)- P^{t_2}(t), \rho(t)\right\rangle_{\mathcal{L}_2} dt+\int_{t_2}^T\left\langle Q^{t_1}(t)-Q^{t_2}(t), \psi(t)\right\rangle_{\mathcal{L}_2} dt=0.
\end{equation}
If $\psi(\cdot)=0$, then $ P^{t_1}(t)= P^{t_2}(t)$  a.e. $t\in[0,T]$. Similarly, if $\rho(\cdot)=0$, then $Q^{t_1}(t)=Q^{t_2}(t)$ a.e. $t\in[0,T]$.
This implies  \eqref{the time consistency of t1 and t2}.

For any $t\in[0,T]$, let $P(t)= P^{0}(t)$, $Q(t)= Q^{0}(t)$. From \eqref{the time consistency of t1 and t2}, we have
\begin{equation*}
\left(P^{t_0}(t),Q^{t_0}(t)\right)=(P(t),Q(t)),\quad t\in[t_0,T].
\end{equation*}
This, together with \eqref{ito formula-X-P}, implies that
\begin{equation}\label{after P and Q}
\begin{aligned}
&\langle P_T, X(T)\rangle_{\mathcal{L}_2}-\left\langle\zeta^{t_0}, X_0\right\rangle_{\mathcal{L}_2}\\
 = &\int_{t_0}^T\langle g(t), X(t) \rangle_{\mathcal{L}_2}dt+\int_{t_0}^T\langle P(t),z(t)\rangle_{\mathcal{L}_2}dt+\int_{t_0}^T \langle Q(t),w(t)\rangle_{\mathcal{L}_2} dt,\\
(z(\cdot),&w(\cdot),X_0)\in L_{\mathbb{A}}^1([t_0,T];\mathcal{L}_2(L^2(\mathscr{C})))\times L^2_{\mathbb{A}}([t_0,T];\mathcal{L}_2(L^2(\mathscr{C})))\times\mathcal{L}_2(L^2(\mathscr{C})).
\end{aligned}
\end{equation}

Finally, we show that $\zeta^{t}=P(t)$ a.e. $t\in[0,T]$. For fix any $\kappa\in \mathcal{L}_2(L^2(\mathscr{C}_{t_1}))$, let $t_0=t_1$, $z(\cdot)=0$, $w(\cdot)=0$, and $X_0=(t_2-t_1)\kappa$ in \eqref{QSDE-X},  we obtain that 
\begin{equation}\label{the final inner product-1}
\left\langle P_T,(t_2-t_1)\kappa\right\rangle_{\mathcal{L}_2}-\left\langle \zeta^{t_1}, (t_2-t_1)\kappa \right\rangle_{\mathcal{L}_2}=\int_{t_1}^{T}\langle g(t), (t_2-t_1)\kappa\rangle_{\mathcal{L}_2} dt.
\end{equation}
Choosing $t_0=t_1$, $z(t)=\chi_{[t_1,t_2]}(t)\kappa$, $w(\cdot)=0$, and $X_0=0$ in \eqref{QSDE-X}. From \eqref{after P and Q}, we find that
\begin{equation}\label{the final inner product-2}
\begin{aligned}
&\langle P_T,(t_2-t_1)\kappa \rangle_{\mathcal{L}_2}\\
=&\int_{t_1}^{t_2}\langle P(t),\kappa\rangle_{\mathcal{L}_2}dt+\int_{t_1}^{t_2}\langle g(t),(t-t_1)\kappa\rangle_{\mathcal{L}_2} dt+\int_{t_2}^{T}\langle g(t),(t_2-t_1)\kappa\rangle_{\mathcal{L}_2} dt.
\end{aligned}
\end{equation}
From \eqref{the final inner product-1} and \eqref{the final inner product-2}, we obtain that
\begin{equation}\label{the final inner of step 1 of transpostion solution}
 \langle\zeta^{t_1},\kappa \rangle_{\mathcal{L}_2}=\frac{1}{t_2-t_1}\int_{t_1}^{t_2}\langle P(t),\kappa\rangle_{\mathcal{L}_2}dt +\frac{1}{t_2-t_1}\int_{t_1}^{t_2}\langle g(t),(t-t_1)\kappa\rangle_{\mathcal{L}_2} dt-\int_{t_1}^{t_2}\langle g(t),\kappa\rangle_{\mathcal{L}_2} dt.
\end{equation}
It is easy to show that
\begin{equation}\label{the final inner of transpostion solution}
\lim_{t_2\to t_1+}\left\{\frac{1}{t_2-t_1}\int_{t_1}^{t_2}\langle g(t), (t-t_1)\kappa\rangle_{\mathcal{L}_2} dt-\int_{t_1}^{t_2}\langle g(t), \kappa\rangle_{\mathcal{L}_2} dt\right\}=0.
\end{equation}
Combined \eqref{the final inner of step 1 of transpostion solution} and \eqref{the final inner of transpostion solution}, we find that, for any $t_1\in[0,T)$, 
\begin{equation*}
  \lim_{t_2\to t_1+}\frac{1}{t_2-t_1}\int_{t_1}^{t_2}\langle P(t), \kappa\rangle_{\mathcal{L}_2} dt=\langle \zeta^{t_1}, \kappa\rangle_{\mathcal{L}_2}, \quad \kappa \in \mathcal{L}_2(L^2(\mathscr{C}_{t_1})).
\end{equation*}
Choose $\kappa:=\zeta^{t_1}-P(t_1)$, for any $t_1\in[0,T)$,
\begin{equation}\label{the final-2-1}
 \lim_{t_2\to t_1+}\frac{1}{t_2-t_1}\int_{t_1}^{t_2}\langle P(t), \zeta^{t_1}-P(t_1)\rangle_{\mathcal{L}_2} dt=\langle\zeta^{t_1}, \zeta^{t_1}-P(t_1)\rangle_{\mathcal{L}_2}.
\end{equation}
By \cite[Lemma 4.14]{L.Z-2020}, there is a monotonic sequence $\{h_n\}_{n=1}^\infty$ of positive numbers with $\lim\limits_{n\to\infty}h_n=0$ such that for any $t_1\in[0,T)$,
\begin{equation}\label{the final-2-2}
\lim_{h_n\to 0}\frac{1}{h_n}\int_{t_1}^{t_1+h_n}\left\langle P(t),\zeta^{t_1}-P(t_1)\right\rangle_{\mathcal{L}_2} dt=\langle P(t_1),\zeta^{t_1}-P(t_1)\rangle_{\mathcal{L}_2}.
\end{equation}
From \eqref{the final-2-1} and \eqref{the final-2-2}, we deduce that
\begin{equation}\label{the fianl}
\langle\zeta^{t_1}, \zeta^{t_1}-P(t_1)\rangle_{\mathcal{L}_2}= \left\langle P(t_1), \zeta^{t_1}-P(t_1)\right\rangle_{\mathcal{L}_2},\quad t_1\in[0,T).
\end{equation}
Then $\|\zeta^{t_1}-P(t_1)\|_{\mathcal{L}_2}^2=0$ for any $t_1\in[0,T)$. It implies that $\zeta^{t}=P(t)$ a.e. $t\in[0,T]$.
Thus, we obtain that 
\begin{equation}\label{inner product of transposition solution}
\begin{aligned}
  & \langle P_T, X(T)\rangle_{\mathcal{L}_2}-\langle P(t_0), X_0\rangle_{\mathcal{L}_2}\\
=&\int_{t_0}^T\langle P(t), z(t)  \rangle_{\mathcal{L}_2}dt +\int_{t_0}^T \langle Q(t), w(t)  \rangle_{\mathcal{L}_2} dt+\int_{t_0}^T\langle g(t), X(t)\rangle_{\mathcal{L}_2} dt.
\end{aligned}
\end{equation}

Let $z(\cdot)=0$ and $w(\cdot)=0$ in \eqref{QSDE-X}. 
 For any $t,t'\in[t_0,T]$, we have 
\begin{equation*}
 \langle P(t)-P(t'), X_0\rangle_{\mathcal{L}_2}=\left\langle \int_{t}^{t'} g(s) ds, X(t) \right\rangle_{\mathcal{L}_2}.
\end{equation*}
Then $\lim\limits_{t'\to t} |\langle P(t)-P(t'),X_0\rangle_{\mathcal{L}_2}|=0,$ that is, 
\begin{equation*}
\textrm{(w)-}\lim\limits_{t'\to t}P(t')=P(t),\ \textrm{ in}\ \mathcal{L}_2(L^2(\mathscr{C})).
\end{equation*}
 Hence, $P(\cdot)\in C_\mathbb{A}([t_0,T];\mathcal{L}_2(L^2(\mathscr{C})))$. Clearly, $(P(\cdot), Q(\cdot))\in C_\mathbb{A}([t_0,T];\mathcal{L}_2(L^2(\mathscr{C})))\times L^2_\mathbb{A}([t_0,T];$\ $\mathcal{L}_2(L^2(\mathscr{C})))$ satisfies that 
\begin{equation}\label{the estimate of P,Q}
\|(P(\cdot), Q(\cdot))\|_{C_\mathbb{A}([t_0,T];\mathcal{L}_2)\times L^2_\mathbb{A}([t_0,T];\mathcal{L}_2)}\leq\mathcal{C}\left(\|g\|_{L^1_\mathbb{A}([t_0,T];\mathcal{L}_2)}+\|P_T\|_{\mathcal{L}_2}\right).
\end{equation}
The uniqueness of the transposition solution to \eqref{BQSDE-P-linear} is obvious.

\textbf{Step 2.} We investigate the transposition solution to \eqref{BSDE-P} with $P_T\in \mathcal{L}_2(L^2(\mathscr{C}_T))$, $\mathbb{H}\in L_\mathbb{A}^\infty([t_0,T];\mathcal{L}_2(L^2(\mathscr{C})))$. 
Consider the transposition solution to the following BQSDE:
\begin{equation}\label{BQSDE-P-3.1}
\left\{
 \begin{aligned}
&dP=f(t,P(t),Q(t))dt+Q(t)dW(t),\quad \textrm{in}\ [0,T),\\
&P(T)=P_T,
 \end{aligned}
\right.
\end{equation}
where
\begin{equation*}
  f(t,P,Q)=-PD-D^*P-F^*Q\Upsilon-Q\Upsilon F-F^*PF+\mathbb{H}.
\end{equation*}
Since $D,F$ are bounded linear operators on $L^2(\mathscr{C})$, $f(\cdot,\cdot,\cdot)$ satisfies the Lipschitz condition, i.e.,
for any $(P,Q)$ and $(\widetilde{P},\widetilde{Q})$, there exists a constant $\mathcal{C}$ such that 
\begin{equation}\label{Lipschitz condition of f}
 \left\|f(t,P,Q)-f\left(t,\widetilde{P},\widetilde{Q}\right)\right\|_{\mathcal{L}_2}\leq \mathcal{C}\left(\left\|P-\widetilde{P}\right\|_{\mathcal{L}_2}+\left\|Q-\widetilde{Q}\right\|_{\mathcal{L}_2}\right),\ t\in[0,T].
\end{equation}
Similar to the classical method \cite[Theorem 4.19]{L.Z-2020}, it is clear that $(P(\cdot),Q(\cdot))\in C_\mathbb{A}([0,T];\mathcal{L}_2(L^2(\mathscr{C})))$ $\times L_\mathbb{A}^2([0,T];\mathcal{L}_2(L^2(\mathscr{C})))$ is the unique transposition solution to \eqref{BQSDE-P-3.1}. Similar to \eqref{the estimate of P,Q}, we can derive \eqref{estimate of the transposition solution} by combining \eqref{Lipschitz condition of f}. 


For $x_1(\cdot),x_2(\cdot)\in C_\mathbb{A}([0,T];L^2(\mathscr{C}))$, define $\textbf{T}(\cdot)$ by $\textbf{T}(\cdot):=x_1(\cdot)\otimes x_2(\cdot)$. For any $y\in L^2(\mathscr{C})$, $\textbf{T}(\cdot)y=\langle x_2(\cdot),y \rangle x_1(\cdot)$.
 Clearly, $\textbf{T}(\cdot)\in \mathcal{L}_2(L^2(\mathscr{C}))$. 
 By fermion It\^{o}'s formula \cite{A.H-2}, we have 
\begin{equation}\label{Ito-x1 and x2}
\begin{aligned}
d\textbf{T} 
=&\left(dx_1\right)\otimes x_2+x_1\otimes \left(dx_2\right)+dx_1\otimes dx_2\\
=&\left\{Dx_1 \otimes x_2 +x_1 \otimes Dx_2+u_1\otimes x_2+x_1\otimes u_2+\left(Fx_1 +v_1\right)\otimes \left(Fx_2 +v_2\right)\right\}dt\\
&+\left\{x_1 \otimes \Upsilon\left(Fx_2+v\right)\right\}dW(t)+dW(t)\left\{\Upsilon\left(Fx_1 +v_1\right)\otimes x_2\right\}.
\end{aligned}
\end{equation}
It is easy to see that
\begin{equation}\label{short of DTF}
\left\{
\begin{aligned}
&Fx_1  \otimes Fx_2 =F \textbf{T}  F^*,\\
&  Dx_1  \otimes x_2  +x_1  \otimes D x_2 =D\textbf{T} +\textbf{T}  D^*,\\
&x_1  \otimes\Upsilon Fx _2+\Upsilon Fx_1 \otimes x_2  =\textbf{T}  F^*\Upsilon+\Upsilon F\textbf{T}.
\end{aligned}
\right.
\end{equation}
Combining \eqref{Ito-x1 and x2} and \eqref{short of DTF}, we can show that $\textbf{T} $ solves the following equation:
\begin{equation}\label{FSDE-T}
\left\{
\begin{aligned}
&d\textbf{T} =\alpha  dt+\beta  dW(t)+dW(t)\gamma ,\quad \textrm{in}\ (t_0,T],\\
&\textbf{T} (t_0)=\xi_1 \otimes \xi_2,
\end{aligned}
\right.
\end{equation}
where
\begin{equation*}
\left\{
\begin{aligned}
\beta =&x_1 \otimes \Upsilon(Fx_2  +v_2) ,\quad\quad \gamma =\Upsilon (Fx_1 +v_1)\otimes x_2,\\
\alpha =&D\textbf{T} +\textbf{T}  D^*+F\textbf{T}  F^*+u_1\otimes x_2 +x_1  \otimes u_2+F x_1  \otimes v_2+v_1\otimes Fx_2  +v_1\otimes v_2.
\end{aligned}
\right.
\end{equation*}

Since the pair $(P(\cdot),Q(\cdot))$ is the transposition solution to \eqref{BSDE-P}, together with \eqref{FSDE-T}, we obtain that
\begin{equation}\label{between T and P in L-2}
\begin{aligned}
&\langle P(T),\textbf{T}(T)\rangle_{\mathcal{L}_2}-\int_{t_0}^T \langle \mathbb{H}(t),\textbf{T}(t)\rangle_{\mathcal{L}_2}dt\\
=&\langle P(t_0),\textbf{T}(t_0)\rangle_{\mathcal{L}_2}+ \int_{t_0}^{T}\left\langle P(t),\alpha(t)\right\rangle_{\mathcal{L}_2}dt -  \int_{t_0}^{T}\left\langle F(t)^*P(t)F(t),\textbf{T}(t)\right\rangle_{\mathcal{L}_2} dt\\
&-\int_{t_0}^T\langle P(t)D(t)+D^*(t)P(t)+F^*(t)Q(t)\Upsilon+Q(t)\Upsilon F(t),\textbf{T}(t)\rangle_{\mathcal{L}_2} dt\\
&+\left\langle\int_{t_0}^TQ(t)dW(t),\int_{t_0}^T\beta(t)dW(t)+\int_{t_0}^TdW(t)\gamma(t)\right\rangle_{\mathcal{L}_2}.
\end{aligned}
\end{equation}
By the definition of $\mathcal{L}_2(L^2(\mathscr{C}))$ and recalling that $\textbf{T}(\cdot)=x_1(\cdot)\otimes x_2(\cdot)$, one has
\begin{equation}\label{L^2(C) denote L_2(L^2(C))}
\langle P(T),\textbf{T}(T) \rangle_{\mathcal{L}_2}=\textrm{tr}(P(T)^*\textbf{T}(T))=\sum_{j=1}^{\infty}\langle P(T)e_j,\textbf{T}(T)e_j\rangle=\langle P(T)x_2(T),x_1(T)\rangle,
\end{equation}
where $\{e_j\}_{j=1}^\infty$ is an orthonormal basis of $L^2(\mathscr{C})$. 
 Similarly, 
we have
\begin{equation}
\left\{
\begin{aligned}
&\langle P(t_0),\textbf{T}(t_0)\rangle_{\mathcal{L}_2}=\langle P(t_0)\xi_2, \xi_1\rangle,\\ 
& \int_{t_0}^T \langle \mathbb{H}(t),\textbf{T}(t)\rangle_{\mathcal{L}_2}dt=\int_{t_0}^T \langle \mathbb{H}(t)x_2(t), x_1(t)\rangle dt,\\
&\int_{t_0}^{T}\left\langle P(t),\alpha(t)\right\rangle_{\mathcal{L}_2}dt=\int_{t_0}^{T}\big\{\left\langle Px_2, Dx_1\right\rangle+\langle PDx_2, x_1\rangle+\langle PFx_2, Fx_1\rangle \\
&\indent\indent\indent+\langle Px_2, u_1\rangle+\langle Pu_2, x_1\rangle+\langle Pv_2, Fx_1\rangle+\langle PFx_2, v_1\rangle+\langle Pv_2, v_1\rangle \big\}dt,\\
&\int_{t_0}^{T}\left\langle F^*PF+PD+D^*P+F^*Q\Upsilon+Q\Upsilon F, \textbf{T}\right\rangle_{\mathcal{L}_2} dt\\
&=\int_{t_0}^{T}\langle PFx_2, Fx_1\rangle+\langle Px_2, Dx_1\rangle+\langle PDx_2, x_1\rangle+\langle Q\Upsilon x_2, F x_1\rangle+\langle Q\Upsilon F x_2, x_1\rangle dt.
\end{aligned}\right.
\end{equation}
Next, we analyze $ \langle \int_{t_0}^T Q(t)dW(t), \int_{t_0}^T\beta(t)dW(t)\rangle_{\mathcal{L}_2}$.
Since $\Upsilon$ is unitary, we have
\begin{equation}\label{inner product of diffusion-1-transposition}
\begin{aligned}
    \left\langle \int_{t_0}^T Q(t)dW(t),\int_{t_0}^T\beta(t)dW(t)\right\rangle_{\mathcal{L}_2}
 = &\sum_{j=1}^{\infty}\left\langle\int_{t_0}^T Q(t)dW(t)e_j, \int_{t_0}^T\beta(t)dW(t)e_j\right\rangle\\
 =&\sum_{j=1}^{\infty} \int_{t_0}^T\langle Q(t)\Upsilon e_j, \beta(t)\Upsilon e_j\rangle dt\\
     =& \int_{t_0}^T\langle Q(t)\Upsilon\left\{F(t)x_2(t)+v_2(t)\right\}, x_1(t)\rangle dt.
\end{aligned}
\end{equation}
Similarly, one has that
\begin{equation}\label{inner product of diffusion-2-transposition}
 \left\langle \int_{t_0}^T Q(t)dW(t), \int_{t_0}^T dW(t)\gamma(t)\right\rangle_{\mathcal{L}_2}
=\int_{t_0}^T\langle Q(t)\Upsilon x_2(t), F(t)x_1(t)+v_1(t)\rangle dt. 
\end{equation}
From \eqref{between T and P in L-2}-\eqref{inner product of diffusion-2-transposition}, we show that the pair $(P(\cdot),Q(\cdot))$ satisfies \eqref{defn-the transposition solution}. Then $(P(\cdot),Q(\cdot))$ is the unique transposition solution to \eqref{BSDE-P} in sense of Definition \ref{The transposition solution}.
\end{proof}

\subsection{The relaxed transposition solution to BQSDEs}\label{The realxed transpsotion solution of BQSDEs}
\indent\indent
This subsection is devoted to the proof of the main result in this paper, i.e.,  Theorem \ref{Thm the relaxed transposition}. 
\begin{proof}[\textbf{The proof of Theorem \ref{Thm the relaxed transposition}}]
First, we prove the existence of relaxed transposition solution to \eqref{BSDE-P}. 
The following proof is divided into serval steps.

\textbf{Step 1.} In this step, we introduce a suitable approximation to \eqref{BSDE-P}.
Recall that $\{e_j\}_{j=1}^\infty$ is an orthonormal basis of $L^2(\mathscr{C})$. Let $\{\Gamma^j\}_{j=1}^\infty$ be the standard projection operator from $L^2(\mathscr{C})$ onto its subspace span$\{e_1,e_2,\cdots,e_n\}$, that is,
$\Gamma^nh=\sum_{j=1}^n\langle e_j, h\rangle e_j,$ for any $h\in L^2(\mathscr{C}).$
 Write
\begin{equation}\label{PT-H of the relaxed transposition solution}
\mathbb{H}^n:=\Gamma^n\mathbb{H},\quad P^n_T:=\Gamma^n P_T.
\end{equation}
Obviously, $\mathbb{H}^n\in L_\mathbb{A}^1([0,T];\mathcal{L}_2(L^2(\mathscr{C})))$, $P_T^n\in \mathcal{L}_2(L^2(\mathscr{C}_T))$, and
\begin{equation}\label{Linear-HS}
\|\mathbb{H}^n\|_{L_\mathbb{A}^1([0,T];\mathcal{L})}+ \|P_T^n\|_{\mathcal{L}}
\leq\mathcal{C}\left(\|\mathbb{H}\|_{L_\mathbb{A}^1([0,T];\mathcal{L})}+ \|P_T\|_{\mathcal{L}}\right).
\end{equation}
Here and henceforth $\mathcal{C}$ denotes a generic constant, independent  of $n$.

Next, we consider the following $\mathcal{L}_2(L^2(\mathscr{C}))$-valued BQSDE:
\begin{equation}\label{BQSDE-HS}
\left\{
\begin{aligned}
&dP^n=-\left\{D(t)^*P^n(t)+P^n(t) D(t)+Q^n(t)\Upsilon F(t)+F(t)^*Q^n(t)\Upsilon\right.\\
&\indent\indent\left.+ F(t)^*P^n(t)F(t)-\mathbb{H}^n(t)\right\}dt+Q^n(t) dW(t),  \quad\rm{in}\ [0,T),\\
&P^n(T)=P_T^n.
\end{aligned}
\right.
\end{equation}
For each $n\in\mathbb{N}$, the equation \eqref{BQSDE-HS} can be regarded as an approximation of \eqref{BSDE-P}. In the rest of the proof, we shall construct the desired solution to \eqref{BSDE-P} by means of the solution to \eqref{BQSDE-HS}.

By Theorem \ref{Thm-the transposition solution}, the equation \eqref{BQSDE-HS} has a unique transposition solution
\begin{equation}\label{the first estimate of solution to BQSDE-HS}
(P^n(\cdot),Q^n(\cdot))\in C_{\mathbb{A}}([t_0,T];\mathcal{L}_2(L^2(\mathscr{C})))\times L_{\mathbb{A}}^2([t_0,T];\mathcal{L}_2(L^2(\mathscr{C}))),
\end{equation}
in the sense of Definition \ref{The transposition solution}. 
Therefore, for any $t_0\in[0,T]$, $\xi_1,\xi_2\in L^2(\mathscr{C}_{t_0})$, $u_1, u_2, v_1, v_2 \in L_{\mathbb{A}}^2([t_0,T];L^2(\mathscr{C}))$,  it holds that
\begin{equation}\label{inner product of Step 1 of The relaxed transposition}
\begin{aligned}
&\left\langle P_T^nx_2(T), x_1(T) \right\rangle-\int_{t_0}^T\langle\mathbb{H}^n(t)x_2(t), x_1(t)\rangle dt\\
=&\langle P^n(t_0)\xi_2, \xi_1\rangle+\int_{t_0}^{T}\langle P^n(t)x_2(t), u_1(t) \rangle dt+\int_{t_0}^{T}\langle P^n(t)u_2(t), x_1(t)\rangle dt\\
& +\int_{t_0}^{T}\langle P^n(t) F(t) x_2(t), v_1(t)\rangle dt +\int_{t_0}^{T}\langle P^n(t)v_2(t), F(t) x_1(t)+ v_1(t)\rangle dt\\
&+\int_{t_0}^{T}\langle Q^n(t)\Upsilon v_2(t),x_1(t)\rangle dt +\int_{t_0}^{T}\left\langle Q^n(t)\Upsilon x_2(t), v_1(t)\right\rangle dt,
\end{aligned}
\end{equation}
where  
$x_1(\cdot)$ and $x_2(\cdot)$ are the solution to \eqref{QSDE-Forward-1} and \eqref{QSDE-Forward-2}, respectively.
And the variational equality \eqref{inner product of Step 1 of The relaxed transposition} is viewed as an approximation of \eqref{Def of relaxed transposition solution}.

\textbf{Step 2.} In this step, we take $n\to\infty$ in \eqref{inner product of Step 1 of The relaxed transposition}. For this purpose, we need to establish some a priori estimates for $P^n(\cdot)$ and $Q^n(\cdot)$.

Let $u_1=v_1=0$ in \eqref{QSDE-Forward-1} and $u_2=v_2=0$ in \eqref{QSDE-Forward-2}. 
From \eqref{inner product of Step 1 of The relaxed transposition}, we obtain that, for any $t_0\in [0,T]$ and $\xi_1,\xi_2 \in L^2(\mathscr{C}_{t_0})$,
\begin{equation}\label{the first inner product of u_1=v_1=u_2=v_2=0}
\begin{aligned}
&\langle P_T^nU(T,t_0)\xi_2, U(T,t_0)\xi_1\rangle-\int_{t_0}^T\langle\mathbb{H}^n(t)U(t,t_0)\xi_2, U(t,t_0)\xi_1 \rangle dt
=\langle P^n(t_0)\xi_2, \xi_1 \rangle.
\end{aligned}
\end{equation}
Hence, for any $\xi_1,\xi_2\in L^2(\mathscr{C}_{t_0})$,
\begin{align*}
&\left\langle U^*(T,t_0)P_T^nU(T,t_0)\xi_2, \xi_1 \right\rangle-\int_{t_0}^T\langle U^*(t,t_0)\mathbb{H}^n(t)U(t,t_0)\xi_2, \xi_1\rangle dt
=\langle P^n(t_0)\xi_2, \xi_1 \rangle.
\end{align*}
Therefore, we have
\begin{equation}\label{P-U-H}
\begin{aligned}
&m\left(U^*(T,t_0)P_T^nU(T,t_0)\xi_2-\int_{t_0}^T U^*(t,t_0)\mathbb{H}^n(t)U(t,t_0)\xi_2 dt\bigg|L^2(\mathscr{C}_{t_0})\right)\\
&\indent\indent=P^n(t_0)\xi_2,\quad  t_0\in [0,T],\ \xi_2\in L^2(\mathscr{C}_{t_0}).
\end{aligned}
\end{equation}
From \eqref{Linear-HS} and \eqref{the first inner product of u_1=v_1=u_2=v_2=0}, we conclude that
\begin{equation}\label{the first estimate of Pn-xi1-xi2}
\begin{aligned}
|\langle P^n(t_0)\xi_2, \xi_1\rangle|\leq& \mathcal{C}\left(\|P_T^n\|_{\mathcal{L}}+\|\mathbb{H}^n\|_{L^1_\mathbb{A}([t_0,T];\mathcal{L})}\right)\|\xi_1\|_{2}\|\xi_2\|_{2}\\
\leq &\mathcal{C}\left(\|P_T\|_{\mathcal{L}}+\|\mathbb{H}\|_{L^1_\mathbb{A}([0,T];\mathcal{L})}\right)\|\xi_1\|_{2}\|\xi_2\|_{2},
\end{aligned}
\end{equation}
where $\mathcal{C}$ does not depend on $n$ and $t_0$.
For the above $P^n(t_0)$, we can find a $\xi_{2,n}\in L^2(\mathscr{C}_{t_0})$ with $\|\xi_{2,n}\|_{2}=1$ such that
\begin{equation}\label{Pn-xi2}
\|P^n(t_0)\xi_{2,n}\|_{2}\geq \frac{1}{2}\|P^n(t_0)\|_{\mathcal{L}}.
\end{equation}
Moreover, we can find a $\xi_{1,n}\in L^2(\mathscr{C}_{t_0})$ with  $\|\xi_{1,n}\|_{2}=1$ such that
\begin{equation}\label{Pn-xi1-xi2}
|\langle P^n(t_0)\xi_{2,n}, \xi_{1,n}\rangle|\geq \frac{1}{2}\|P^n(t_0)\xi_{1,n}\|_{2}.
\end{equation}
It follows from that \eqref{the first estimate of Pn-xi1-xi2}-\eqref{Pn-xi1-xi2}, we deduce that, for all $n\in \mathbb{N}$,
\begin{equation}\label{the estimate about Pn}
\|P^n\|_{L_\mathbb{A}^\infty([0,T];\mathcal{L})}\leq \mathcal{C}\left(\|P_T\|_{\mathcal{L}}+\|\mathbb{H}\|_{L^1_\mathbb{A}([0,T];\mathcal{L})}\right).
\end{equation}
By \cite[Theorem 3.17]{Rudin},
we can find  $P\in \mathcal{L}(L^2_\mathbb{A}([0,T];L^2(\mathscr{C})))$ such that
\begin{equation}\label{the estimate about P}
\|P\|_{\mathcal{L}(L^2_\mathbb{A}([0,T];L^2(\mathscr{C})))}\leq \mathcal{C}\left(\|P_T\|_{\mathcal{L}}+\|\mathbb{H}\|_{L^1_\mathbb{A}([0,T];\mathcal{L})}\right),
\end{equation}
 and a subsequence $\big\{P^{n_k^{(1)}}\big\}_{k=1}^\infty\subset\big\{P^n\big\}_{n=1}^\infty$ so that, for any  $u_2\in L^2_\mathbb{A}([0,T];L^2(\mathscr{C}))$,
\begin{equation}\label{weak convergence of Pn to P}
\textrm{(w)-}\lim_{k\to\infty}P^{n^{(1)}_k}u_2=Pu_2, \ \textrm{in}\ L^2_\mathbb{A}([0,T];L^2(\mathscr{C})),
\end{equation}
since $\left(\mathcal{L}_1(L^2_\mathbb{A}([0,T];L^2(\mathscr{C})))\right)^*\cong\mathcal{L}(L^2_\mathbb{A}([0,T];L^2(\mathscr{C})))$.
Similarly, by \cite[Theorem 3.17]{Rudin} again and \eqref{the estimate about Pn}, for each fixed $t_0\in[0,T)$, there exists  an $R^{(t_0)}\in \mathcal{L}(L^2(\mathscr{C}_{t_0}))$ and a subsequence $\big\{P^{n^{(2)}_k}\big\}_{k=1}^\infty\subset\big\{P^{n_k^{(1)}}\big\}_{n=1}^\infty$  such that, for all $\xi\in L^2(\mathscr{C}_{t_0})$,
\begin{equation}\label{weak convergence of Pn to R}
\textrm{(w)-}\lim_{k\to\infty}P^{n^{(2)}_k}(t_0)\xi=R^{(t_0)}\xi, \quad  \textrm{in}\ L^2(\mathscr{C}_{t_0}).
\end{equation}

Next, let $\xi_1=0$, $u_1=0$ and $u_2=v_2=0$ in \eqref{QSDE-Forward-1} and \eqref{QSDE-Forward-2}, respectively. 
From \eqref{inner product of Step 1 of The relaxed transposition} and Lemma \ref{the relationship between solution and coefficients}, we have
\begin{align*}
&\int_{t_0}^{T}\left\langle Q^n(t)\Upsilon U(t,t_0)\xi_2, v_1(t) \right\rangle dt\\
=&\left\langle P_T^nx_2(T), x_1(T)\right\rangle-\int_{t_0}^T\left\langle \mathbb{H}^n(t)x_2(t), x_1(t)\right\rangle dt-\int_{t_0}^{T}\langle P^n(t) F(t) x_2(t), v_1(t) \rangle dt.
\end{align*}
This, together with  \eqref{Linear-HS}, implies that, for any  $v_1\in L^2_\mathbb{A}([0,T];L^2(\mathscr{C})), \xi_2\in L^2(\mathscr{C}_{t_0})$,
\begin{equation}\label{the estimate of Q-1}
\begin{aligned}
 &\left|\int_{t_0}^{T}\left\langle Q^n(t)\Upsilon U(t,t_0)\xi_2, v_1(t)\right\rangle dt\right|\\
\leq& \mathcal{C}\left(\|P_T^n\|_{\mathcal{L}}+\|\mathbb{H}^n\|_{L_\mathbb{A}^1([0,T];\mathcal{L})}\right)\|v_1\|_{L_\mathbb{A}^2([t_0,T];L^2(\mathscr{C}))}\|\xi_2\|_{2}\\
 \leq &\mathcal{C}\left(\|P_T\|_{\mathcal{L}}+\|\mathbb{H}\|_{L_\mathbb{A}^1([0,T];\mathcal{L})}\right)\|v_1\|_{L_\mathbb{A}^2([t_0,T];L^2(\mathscr{C}))}\|\xi_2\|_{2}.
\end{aligned}
\end{equation}
We define two operators $Q_1^{n,t_0}$ and $\widehat{Q}_1^{n,t_0}$ from $L^2(\mathscr{C}_{t_0})$ to $L^2_\mathbb{A}([t_0,T]; L^2(\mathscr{C}))$ as follows:
\begin{equation*}
 \left\{
 \begin{aligned}
&Q_1^{n,t_0}\xi:=Q^n(\cdot)\Upsilon U(\cdot,t_0)\xi,  \ & \xi\in L^2(\mathscr{C}_{t_0}),\\
&\widehat{Q}_1^{n,t_0}\xi:=\Upsilon Q^n(\cdot)^*U(\cdot,t_0)\xi, \ & \xi\in L^2(\mathscr{C}_{t_0}),
\end{aligned}
 \right.
\end{equation*}
where $U(\cdot,t_0)$ is introduced in Lemma \ref{the relationship between solution and coefficients}.
It is easy to check that $Q_1^{n,t_0},\widehat{Q}_1^{n,t_0}\in \mathcal{L}(L^2(\mathscr{C}_{t_0}); L^2_\mathbb{A}([t_0,T];$ $L^2(\mathscr{C})))$. From \eqref{the estimate of Q-1}, we have
\begin{equation}\label{the estimate of Qn1}
\left\|Q_1^{n,t_0}\right\|_{\mathcal{L}(L^2(\mathscr{C}_{t_0}); L^2_\mathbb{A}([t_0,T];L^2(\mathscr{C})))}\leq \mathcal{C}\left(\|P_T\|_{\mathcal{L}}+\|\mathbb{H}\|_{L_\mathbb{A}^1([0,T];\mathcal{L})} \right).
\end{equation}
Similarly, let $u_1=v_1=0$, $\xi_2=0$ and $u_2=0$,  we obtain that
\begin{equation*}
\begin{aligned}
  &\int_{t_0}^{T}\langle Q^n(t)\Upsilon v_2(t), U(t,t_0)\xi_1\rangle dt\\
  =&\left\langle  P_T^nx_2(T), x_1(T)\right\rangle-\int_{t_0}^T\left\langle\mathbb{H}^n(t)x_2(t),  x_1(t)\right\rangle dt-\int_{t_0}^{T}\langle P^n(t)v_2(t), F(t) x_1(t)\rangle dt.
  \end{aligned}
\end{equation*}
Similar to \eqref{the estimate of Qn1}, it holds that
\begin{equation}\label{the estimate of Qn2}
\left\|\widehat{Q}_1^{n,t_0}\right\|_{\mathcal{L}(L^2(\mathscr{C}_{t_0}); L^2_\mathbb{A}([t_0,T];L^2(\mathscr{C})))}\leq \mathcal{C}\left(\|P_T\|_{\mathcal{L}}+\|\mathbb{H}\|_{L_\mathbb{A}^1([0,T];\mathcal{L})} \right).
\end{equation}
By \cite[Theorem 3.17]{Rudin} again, for each $t_0$, there exist two bounded linear operators $Q_1^{t_0}$ and $\widehat{Q}^{t_0}_1$, from $L^2(\mathscr{C}_{t_0})$ to $ L^2_\mathbb{A}([t_0,T];L^2(\mathscr{C}))$, and  sequences $\big\{Q_1^{n_k^{(3)},t_0}\big\}_{k=1}^\infty\subset\big\{Q_1^{n_k^{(2)},t_0}\big\}_{k=1}^\infty$ and $\big\{\widehat{Q}_1^{n_k^{(3)},t_0}\big\}_{k=1}^\infty$ $\subset\big\{\widehat{Q}_1^{n_k^{(2)},t_0}\big\}_{k=1}^\infty$ such that, for any $\xi\in L^2(\mathscr{C}_{t_0})$
\begin{equation}\label{weak convergence Q1}
\left\{
\begin{aligned}
&\textrm{(w)-}\lim_{k\to\infty}Q_1^{n_k^{(3)},t_0}\xi=Q_1^{t_0}\xi, \quad {\rm in}\ L^2_\mathbb{A}([t_0,T];L^2(\mathscr{C})),\\
&\textrm{(w)-}\lim_{k\to\infty} \widehat{Q}_1^{n_k^{(3)},t_0}\xi=\widehat{Q}_1^{t_0}\xi, \quad {\rm in}\ L^2_\mathbb{A}([t_0,T];L^2(\mathscr{C})).
\end{aligned}
 \right.
\end{equation}
From \eqref{the estimate of Qn1}-\eqref{weak convergence Q1}, we have
\begin{equation}\label{the estimate of Q1t0}
\begin{aligned}
&\left\|Q_1^{t_0}\right\|_{\mathcal{L}(L^2(\mathscr{C}_{t_0}); L^2_\mathbb{A}([t_0,T];L^2(\mathscr{C})))}+\left\|\widehat{Q}_1^{t_0}\right\|_{\mathcal{L}(L^2(\mathscr{C}_{t_0}); L^2_\mathbb{A}([t_0,T];L^2(\mathscr{C})))}\\
&\indent\indent\indent \leq \mathcal{C}\left(\|P_T\|_{\mathcal{L}}+\|\mathbb{H}\|_{L_\mathbb{A}^1([0,T];\mathcal{L})} \right).
\end{aligned}
\end{equation}

On the other hand, let $\xi_1=0$, $u_1=0$ in \eqref{QSDE-Forward-1}, and $\xi_2=0$, $v_2=0$ in \eqref{QSDE-Forward-2}. 
From \eqref{inner product of Step 1 of The relaxed transposition}, we obtain that
\begin{equation}\label{inner product of xi1=xi2=u1=v2=0}
\begin{aligned}
&\langle P_T^nx_2(T), x_1(T)\rangle-\int_{t_0}^T\langle \mathbb{H}^n(t)x_2(t), x_1(t)\rangle dt\\
=&\int_{t_0}^{T}\left\langle P^n(t)u_2(t), x_1(t) \right\rangle dt +\int_{t_0}^{T}\left\langle P^n(t)F(t) x_2(t), v_1(t)\right\rangle   dt+\int_{t_0}^{T}\left\langle Q^n(t)\Upsilon x_2(t),v_1(t)\right\rangle dt.
\end{aligned}
\end{equation}
Then we define an operator $Q_2^{n,t_0}:L_\mathbb{A}^2([t_0,T];L^2(\mathscr{C}))\to L_\mathbb{A}^2([t_0,T];L^2(\mathscr{C}))$ by
\begin{equation*}
 (Q_2^{n,t_0}u)(\cdot):= Q^n(\cdot)\Upsilon V(\cdot,t_0)u,\quad   u\in L_\mathbb{A}^2([t_0,T];L^2(\mathscr{C})),
\end{equation*}
where $V(\cdot, t_0)$ is introduced in Lemma \ref{the relationship between solution and coefficients}.
This, together with \eqref{inner product of xi1=xi2=u1=v2=0}, implies that
\begin{equation}\label{inner product about Q2nt0}
\begin{aligned}
&\int_{t_0}^{T}\left\langle (Q_2^{n,t_0} u_2)(t), v_1(t)\right\rangle dt\\
=&\left\langle P_T^nx_2(T), x_1(T)\right\rangle-\int_{t_0}^{T}\langle P^n(t)F(t) x_2(t), v_1(t)\rangle dt\\
& -\int_{t_0}^T\left\langle\mathbb{H}^n(t)x_2(t),  x_1(t)\right\rangle dt-\int_{t_0}^{T}\langle P^n(t) u_2(t), x_1(t)\rangle dt\\
\leq& \mathcal{C} \left(\| P_T\|_{\mathcal{L}}+\|\mathbb{H}\|_{L_\mathbb{A}^1([0,T];\mathcal{L})}\right)
\|v_1\|_{L_\mathbb{A}^2([t_0,T];L^2(\mathscr{C}_T))}\|u_2\|_{L_\mathbb{A}^2([t_0,T];L^2(\mathscr{C}))},\\
&\indent\indent\indent\indent v_1, u_2\in L_\mathbb{A}^2([t_0,T];L^2(\mathscr{C})).
\end{aligned}
\end{equation}
From \eqref{inner product about Q2nt0}, we deduce that
\begin{equation}\label{the estimate of Q2nt0}
\left\|Q_2^{n,t_0}\right\|_{\mathcal{L}(L_\mathbb{A}^2([t_0,T];L^2(\mathscr{C})))}\leq \mathcal{C} \left(\|P_T\|_{\mathcal{L}}+\|\mathbb{H}\|_{L_\mathbb{A}^1([0,T];\mathcal{L})}\right).
\end{equation}
Besides, let $\xi_1=0$, $v_1=0$ in \eqref{QSDE-Forward-1}, and $\xi_2=0$, $u_2=0$ in \eqref{QSDE-Forward-2}, it holds that
\begin{align*}
&\left\langle P_T^nx_2(T), x_1(T) \right\rangle-\int_{t_0}^T\left\langle\mathbb{H}^n(t)x_2(t), x_1(t)\right\rangle dt\\
 =&\int_{t_0}^{T}\left\langle P^n(t)x_2(t), u_1(t) \right\rangle dt +\int_{t_0}^T\left\langle P^n(t)v_2(t), F(t) x_1(t)\right\rangle dt+\int_{t_0}^T\left\langle Q^n(t)\Upsilon v_2(t), x_1(t)\right\rangle dt.
\end{align*}
And, we define a linear operator $\widehat{Q}_2^{n,t_0}:L_\mathbb{A}^2([t_0,T];L^2(\mathscr{C}))\to L_\mathbb{A}^2([t_0,T];L^2(\mathscr{C}))$ by
\begin{equation*}
\left(\widehat{Q}_2^{n,t_0}u\right)(\cdot):=\Upsilon Q^n(\cdot)^*V(\cdot,t_0)u, \quad  u\in L_\mathbb{A}^2([t_0,T];L^2(\mathscr{C})).
\end{equation*}
By a similar argument to \eqref{the estimate of Q2nt0}, we find that
\begin{equation}\label{the estimate of widehat-Q2nt0}
\left\|\widehat{Q}_2^{n,t_0}\right\|_{\mathcal{L}(L_\mathbb{A}^2([t_0,T];L^2(\mathscr{C})))}\leq \mathcal{C} \left(\|P_T\|_{\mathcal{L}}+\|\mathbb{H}\|_{L_\mathbb{A}^1([0,T];\mathcal{L})}\right).
\end{equation}
By \cite[Theorem 3.17]{Rudin} again,
  for each $t_0$, there exist two bounded linear operator $Q_2^{t_0}$ and $\widehat{Q}_2^{t_0}$ on $L_\mathbb{A}^2([0,T];L^2(\mathscr{C}))$  and subsequences $\big\{Q_2^{n_k^{(4)},t_0}\big\}_{k=1}^\infty\subset \big\{Q_2^{n_k^{(3)},t_0}\big\}_{k=1}^\infty$ and $\big\{\widehat{Q}_2^{n_k^{(4)},t_0}\big\}_{k=1}^\infty\subset \big\{\widehat{Q}_2^{n_k^{(3)},t_0}\big\}_{k=1}^\infty$ such that, for all $u\in L_\mathbb{A}^2([0,T];L^2(\mathscr{C}))$
\begin{equation}\label{weak convergence of Q2}
\left\{
\begin{aligned}
&\textrm{(w)-}\lim_{k\to\infty}Q_2^{n_k^{(4)},t_0}u=Q_2^{t_0}u, \quad  {\textrm{in}}\ L^2_\mathbb{A}([t_0,T];L^2(\mathscr{C})),\\
&\textrm{(w)-}\lim_{k\to\infty} \widehat{Q}_2^{n_k^{(4)},t_0}u=\widehat{Q}_2^{t_0}u,\quad {\textrm{in}}\ L^2_\mathbb{A}([t_0,T];L^2(\mathscr{C})).
\end{aligned}
 \right.
\end{equation}
From \eqref{the estimate of Q2nt0} and \eqref{the estimate of widehat-Q2nt0}, we obtain that
\begin{equation}\label{the estimate of Q2-t0}
\left\|Q_2^{t_0}\right\|_{\mathcal{L}(L^2_\mathbb{A}([t_0,T];L^2(\mathscr{C})))}+\left\|\widehat{Q}_2^{t_0}\right\|_{\mathcal{L}(L^2_\mathbb{A}([t_0,T];L^2(\mathscr{C})))}
\leq\mathcal{C}\left(\|P_T\|_{\mathcal{L}}+\|\mathbb{H}\|_{L_\mathbb{A}^1([0,T];\mathcal{L})}\right).
\end{equation}

Next, we choose $\xi_1=0$ and $u_1=0$ in \eqref{QSDE-Forward-1}, and $\xi_2=0$ and $u_2=0$ in \eqref{QSDE-Forward-2}. 
From \eqref{inner product of Step 1 of The relaxed transposition}, we obtain that
\begin{equation}\label{inner product of xi1=xi2=u1=u2=0}
\begin{aligned}
&\langle P_T^nx_2(T),  x_1(T) \rangle-\int_{t_0}^T\left\langle\mathbb{H}^n(t)x_2(t), x_1(t)\right\rangle dt\\
=&\int_{t_0}^{T}\langle Q^n(t)\Upsilon v_2(t),  x_1(t)\rangle dt+\int_{t_0}^{T}\left\langle Q^n(t)\Upsilon x_2(t), v_1(t)\right\rangle dt\\
& +\int_{t_0}^{T}\left\langle P^n(t) F(t) x_2(t), v_1(t)\right\rangle dt +\int_{t_0}^{T}\left\langle P^n(t)v_2(t), F(t) x_1(t)+ v_1(t)\right\rangle dt.
\end{aligned}
\end{equation}
We define a semi-bilinear functional $\mathcal{B}_{n,t_0}(\cdot,\cdot)$ on $(L^2_\mathbb{A}([t_0,T];L^2(\mathscr{C})))^2$ as follows: 
\begin{equation}\label{definition of bilinear functional B}
\begin{gathered}
\mathcal{B}_{n,t_0}(v_1,v_2):=\int_{t_0}^{T}\left\langle Q^n(t)\Upsilon x_2(t), v_1(t)\right\rangle dt +\int_{t_0}^{T}\left\langle Q^n(t)\Upsilon v_2(t), x_1(t) \right\rangle dt,\\
  v_1,v_2\in L^2_\mathbb{A}([t_0,T];L^2(\mathscr{C})).
\end{gathered}
\end{equation}
It is easy to check that $\mathcal{B}_{n,t_0}(\cdot,\cdot)$ is a bounded semi-bilinear functional. From \eqref{inner product of xi1=xi2=u1=u2=0}, we have
\begin{equation}\label{inner product about Bnk}
\begin{aligned}
 &\mathcal{B}_{n_k^{(1)},t_0}(v_1,v_2)\\
 =&\left\langle P_T^{n_k^{(1)}}x_2(T), x_1(T) \right\rangle-\int_{t_0}^{T}\left\langle P^{n_k^{(1)}}(t) F(t) x_2(t), v_1(t)\right\rangle dt\\
 &- \int_{t_0}^T\left\langle\mathbb{H}^{n_k^{(1)}}(t)x_2(t), x_1(t)\right\rangle dt-\int_{t_0}^{T}\left\langle P^{n_k^{(1)}}(t)v_2(t), F(t) x_1(t)+ v_1(t)\right\rangle dt.
 \end{aligned}
\end{equation}
It is easy to verify that
\begin{equation*}
  \left\{
\begin{aligned}
 & \lim_{k\to\infty}\left\langle P_T^{n_k^{(1)}}x_2(T), x_1(T) \right\rangle=\langle P_Tx_2(T), x_1(T)\rangle,\\
 & \lim_{k\to\infty} \int_{t_0}^T\left\langle \mathbb{H}^{n_k^{(1)}}(t)x_2(t), x_1(t)\right\rangle dt=\int_{t_0}^T\left\langle\mathbb{H}(t)x_2(t), x_1(t)\right\rangle dt,\\
 &\lim_{k\to\infty}\int_{t_0}^{T}\left\langle P^{n_k^{(1)}}(t) F(t) x_2(t), v_1(t)\right\rangle dt=\int_{t_0}^{T}\langle P(t) F(t) x_2(t), v_1(t)\rangle dt,\\
& \lim_{k\to\infty}\int_{t_0}^{T}\left\langle P^{n_k^{(1)}}(t)v_2(t), F(t) x_1(t)+ v_1(t)\right\rangle dt=\int_{t_0}^{T}\langle P(t)v_2(t), F(t) x_1(t)+ v_1(t)\rangle dt,
  \end{aligned}
  \right.
\end{equation*}
where $x_1$ (\textit{resp.} $x_2$) is the solution to \eqref{QSDE-Forward-1} (\textit{resp.} \eqref{QSDE-Forward-2}) under the condition that $\xi_1=0$ and $u_1=0$ (\textit{resp.} $\xi_2=0$ and $u_2=0$). 
Combined with \eqref{inner product about Bnk}, we deduce that
\begin{equation}\label{inner product about B}
 \begin{aligned}
\mathcal{B}_{t_0}(v_1,v_2)=&\lim_{k\to\infty}\mathcal{B}_{n_k^{(1)},t_0}(v_1,v_2)\\
=&\left\langle P_Tx_2(T), x_1(T)\right\rangle-\int_{t_0}^{T}\left\langle P(t) F(t) x_2(t),  v_1(t)\right\rangle dt\\
 & - \int_{t_0}^T\left\langle \mathbb{H}(t)x_2(t), x_1(t)\right\rangle dt-\int_{t_0}^{T}\langle P(t)v_2(t), F(t) x_1(t)+ v_1(t)\rangle dt.
 \end{aligned}
\end{equation}
It follows from Lemma \ref{the relationship between solution and coefficients}, \eqref{the estimate about P},  \eqref{definition of bilinear functional B} and \eqref{inner product about B} that
\begin{equation}\label{the estimate of the Bt0}
|\mathcal{B}_{t_0}(v_1,v_2)|\leq \mathcal{C}\left(\|P_T\|_{\mathcal{L}}+\|\mathbb{H}\|_{L^1_\mathbb{A}([0,T];\mathcal{L})}\right)\\
\|v_1\|_{L^2_\mathbb{A}([t_0,T];L^2(\mathscr{C}))}\|v_2\|_{L^2_\mathbb{A}([t_0,T];L^2(\mathscr{C}))}.
\end{equation}
Hence, $\mathcal{B}_{t_0}(\cdot,\cdot)$ is a bounded semi-bilinear functional on $(L^2_\mathbb{A}([t_0,T];L^2(\mathscr{C})))^2$. 
Now, for any fixed $v_1\in L^2_\mathbb{A}([t_0,T];L^2(\mathscr{C}))$, it is easy to check that $\mathcal{B}_{t_0}( v_1,\cdot)$ is a bounded  conjugate linear functional on $L^2_\mathbb{A}([t_0,T];L^2(\mathscr{C}))$. Therefore 
there is a unique $\widetilde{v}_2\in L^2_\mathbb{A}([t_0,T];L^2(\mathscr{C}))$ such that
\begin{equation*}
 \mathcal{B}_{t_0}(v_1,v_2)=\langle \widetilde{v}_2,v_1\rangle_{L^2_\mathbb{A}([t_0,T];L^2(\mathscr{C})),L^2_\mathbb{A}([t_0,T];L^2(\mathscr{C}))},\quad  v_2\in L^2_\mathbb{A}([t_0,T];L^2(\mathscr{C})).
\end{equation*}
Next, we define an operator $\widehat{Q}_3^{t_0}:L^2_\mathbb{A}([t_0,T];L^2(\mathscr{C}))\to L^2_\mathbb{A}([t_0,T];L^2(\mathscr{C}))$ as follows: 
\begin{equation*} 
\widehat{Q}_3^{t_0} v_2=\widetilde{v}_2.
\end{equation*}
 From the uniqueness of $\widetilde{v}_2$, it is clear that $\widehat{Q}_3^{t_0}$ is well-defined. Further, combined with \eqref{the estimate of the Bt0}, we can infer that
\begin{equation}\label{the estimate of widehatQ3t0}
\begin{aligned}
\left\|\widehat{Q}_3^{t_0} v_2\right\|_{L^2_\mathbb{A}([t_0,T];L^2(\mathscr{C}))}=&\left\|\widetilde{v}_2\right\|_{L^2_\mathbb{A}([t_0,T];L^2(\mathscr{C}))}\\
\leq& \mathcal{C}\left(\|P_T\|_{\mathcal{L}}+\|\mathbb{H}\|_{L^1_\mathbb{A}([0,T];\mathcal{L})}\right)\|v_2\|_{L^2_\mathbb{A}([t_0,T];L^2(\mathscr{C}))}.
\end{aligned}
\end{equation}
Thus $\widehat{Q}_3^{t_0}$ is a bounded operator. For any $\alpha,\beta\in\mathbb{C}$ and $v_1,w_1,w_2\in L^2_\mathbb{A}([t_0,T];L^2(\mathscr{C}))$,
\begin{align*}
   &\left\langle \widehat{Q}_3^{t_0}(\alpha w_1+\beta w_2), v_1\right\rangle_{L^2_\mathbb{A}([t_0,T];L^2(\mathscr{C})),L^2_\mathbb{A}([t_0,T];L^2(\mathscr{C}))}\\
 &\indent\indent =\mathcal{B}_{t_0}(v_1,\alpha w_1+\beta w_2)\\
 &\indent\indent= \langle \alpha w_1,v_1\rangle_{L^2_\mathbb{A}([t_0,T];L^2(\mathscr{C})),L^2_\mathbb{A}([t_0,T];L^2(\mathscr{C}))}\\
 &\indent\indent\indent +\langle \beta w_2,v_1\rangle_{L^2_\mathbb{A}([t_0,T];L^2(\mathscr{C})),L^2_\mathbb{A}([t_0,T];L^2(\mathscr{C}))},
\end{align*}
which indicates that
\begin{equation*}
\widehat{Q}_3^{t_0}(\alpha w_1+\beta w_2) =\alpha\widehat{Q}_3^{t_0}w_1+\beta\widehat{Q}_3^{t_0}w_1.
\end{equation*}
Therefore, $\widehat{Q}_3^{t_0}$ is a bounded linear operator on $L^2_\mathbb{A}([t_0,T];L^2(\mathscr{C}))$. Let $Q_3^{t_0}:=\frac{1}{2}\widehat{Q}_3^{t_0}$. 
It follows from \eqref{the estimate of widehatQ3t0} that
\begin{equation}\label{the estimate of Q3-t0}
\left\|Q_3^{t_0}\right\|_{\mathcal{L}(L^2_\mathbb{A}([t_0,T];L^2(\mathscr{C})))}\leq \mathcal{C}\left(\|P_T\|_{\mathcal{L}}+\|\mathbb{H}\|_{L^1_\mathbb{A}([0,T];\mathcal{L})}\right).
\end{equation}

For any $t_0\in[0,T]$, we define two operators $Q^{(t_0)}(\cdot,\cdot,\cdot)$ and $\widehat{Q}^{(t_0)}(\cdot,\cdot,\cdot)$ on  $L^2(\mathscr{C}_{t_0})\times L^1_\mathbb{A}([t_0,T];$ $L^2(\mathscr{C}))\times L^2_\mathbb{A}([t_0,T];L^2(\mathscr{C}))$ as follows:
\begin{equation}\label{definition of Qt0 and widetilde Qt0}
\left\{
\begin{aligned}
&Q^{(t_0)}(\xi,u,v):=Q_1^{t_0}\xi+Q_2^{t_0}u+Q_3^{t_0} v,\\
&\widehat{Q}^{(t_0)}(\xi,u,v):=\widehat{Q}_1^{t_0}\xi+\widehat{Q}_2^{t_0}u+(Q_3^{t_0})^*v,\\
&\forall\; (\xi,u,v)\in L^2(\mathscr{C}_{t_0})\times L^2_\mathbb{A}([t_0,T];L^2(\mathscr{C}))\times L^2_\mathbb{A}([t_0,T];L^2(\mathscr{C})).
\end{aligned}
\right.
\end{equation}
According to the definition of $Q_1^{t_0},Q_2^{t_0}$, $Q_3^{t_0}$ (\textit{resp.} $\widehat{Q}_1^{t_0},\widehat{Q}_2^{t_0}$ and $\widehat{Q}_3^{t_0}$), it is obvious that $Q^{(t_0)}(\cdot,\cdot,\cdot)$ (\textit{resp.} $\widehat{Q}^{(t_0)}(\cdot,\cdot,\cdot)$) is bounded linear operators from $ L^2(\mathscr{C}_{t_0})\times L^2_\mathbb{A}([t_0,T];L^2(\mathscr{C}))\times L^2_\mathbb{A}([t_0,T];L^2(\mathscr{C}))$ to $L^2_\mathbb{A}([t_0,T];L^2(\mathscr{C}))$, and 
\begin{equation*}
Q^{(t_0)}(0,0,\cdot)^*=\widehat{Q}^{(t_0)}(0,0,\cdot).
\end{equation*}

For any $t_0\in [0,T]$, from \eqref{inner product of Step 1 of The relaxed transposition}, \eqref{weak convergence of Pn to P}, \eqref{weak convergence of Pn to R}, \eqref{weak convergence Q1}, \eqref{weak convergence of Q2}, \eqref{definition of bilinear functional B}, \eqref{inner product about B} and \eqref{definition of Qt0 and widetilde Qt0}, we find that, for all $(\xi_1,u_1,v_1),(\xi_2,u_2,v_2)\in L^2(\mathscr{C}_{t_0})\times L^2_\mathbb{A}([t_0,T];L^2(\mathscr{C}))\times L^2_\mathbb{A}([t_0,T];L^2(\mathscr{C}))$,
\begin{equation}\label{inner product of finally of step2}
\begin{aligned}
&\left\langle P_Tx_2(T),x_1(T)\right\rangle-\int_{t_0}^T\langle \mathbb{H}(t)x_2(t),x_1(t)\rangle dt\\
=&\langle R^{(t_0)}\xi_2, \xi_1\rangle+\int_{t_0}^{T}\langle P(t)x_2(t), u_1(t)\rangle dt+\int_{t_0}^{T}\langle P(t)u_2(t), x_1(t) \rangle dt\\
&+\int_{t_0}^{T}\langle P(t)F(t) x_2(t), v_1(t)\rangle dt +\int_{t_0}^{T}\langle P(t)v_2(t), F(t) x_1(t)+ v_1(t)\rangle dt\\
&+\int_{t_0}^{T}\left\langle Q^{(t_0)}(\xi_2,u_2, v_2)(t), v_1(t)\right\rangle dt+\int_{t_0}^{T}\left\langle v_2(t),\widehat{Q}^{(t_0)}(\xi_1,u_1,v_1)(t) \right\rangle dt.
\end{aligned}
\end{equation}
\textbf{Step 3.} In this step, we claim that $P(\cdot)\in C_\mathbb{A}([0,T];\mathcal{L}(L^2(\mathscr{C})))$ and
\begin{equation}\label{the aim of step 3-1}
P(t)=R^{(t)},\quad \textrm{a.e.}\quad t\in[0,T].
\end{equation}
Similar to \eqref{P-U-H}, by \eqref{inner product of finally of step2}, we find that, for any $\xi\in L^2(\mathscr{C}_{t_0})$,
\begin{equation}\label{the first equation about R}
 m\left(U^*(T,t_0)P_TU(T,t_0)\xi-\int_{t_0}^T U^*(t,t_0)\mathbb{H}(t)U(t,t_0)\xi dt\bigg|L^2(\mathscr{C}_{t_0})\right)
=R^{(t_0)}\xi, \  t_0\in [0,T].
\end{equation}

Now we show that $R^{(\cdot)}\xi \in C_\mathbb{A}([t_0,T];L^2(\mathscr{C}))$ for any $\xi\in L^2(\mathscr{C}_{t_0})$. From \eqref{the first estimate of solution to BQSDE-HS}, it remains to show that
\begin{equation}\label{the purpose of step 3-2}
\lim_{n\to\infty}\left\|R^{(\cdot)}\xi-P^n(\cdot)\xi\right\|_{C_\mathbb{A}([t_0,T];L^2(\mathscr{C}))}=0.
\end{equation}
In order to obtain it, from \eqref{P-U-H} and \eqref{the first equation about R}, we have that, for any $t\in[t_0,T]$
\begin{equation*}
\begin{aligned}
\left\|R^{(t)}\xi-P^n(t)\xi\right\|_2\leq& \left\|U^*(T,t)\left\{P_T-P_T^n\right\}U(T,t)\xi\right\|_2\\
&+\left\|\int_{t}^T U^*(s,t)\left\{\mathbb{H}(s)-\mathbb{H}^n(s)\right\}U(s,t)\xi ds\right\|_2.
\end{aligned}
\end{equation*}
By Lemma \ref{the relationship between solution and coefficients}, we see that for any $\varepsilon_1>0$, there is a $\delta_1>0$ so that for all $\tau\in [t_0,T]$ and $\sigma\in[\tau,\tau+\delta_1]$,
\begin{equation}\label{U-tau-sigma}
\|U(t,\tau)\xi-U(t,\sigma)\xi\|_{2}<\varepsilon_1,\quad  t\in [\sigma, T].
\end{equation}
Now, let us choose a monotonically increasing sequence $\{t_i\}_{i=1}^{N_1}\subset [0,T]$ for $N_1$ being sufficiently large such that $t_{i+1}-t_i\leq \delta_1$ with $t_1=t$ and $t_{N_1}=T$, and 
\begin{equation}\label{a estimate of H in finite interval}
\int_{t_i}^{t_{i+1}}\|\mathbb{H}(s)\|_{\mathcal{L}}ds<\varepsilon_1,\ \textrm{ for all}\ i=1,\cdots, N_1-1.
\end{equation}
For any $t\in(t_i,t_{i+1}]$, we conclude that
\begin{equation}\label{the estimate of H-Hn}
 \begin{aligned}
 &\left\|\int_{t}^TU^*(s,t)\{\mathbb{H}(s)-\mathbb{H}^n(s)\}U(s,t)\xi ds\right\|_{2}\\
&\leq\left\|\int_{t_i}^TU^*(s,t)\{\mathbb{H}(s)-\mathbb{H}^n(s)\}U(s,t_i)\xi ds\right\|_{2}\\
& \indent+ \left\|\int_{t}^{t_i}U^*(s,t)\{\mathbb{H}(s)-\mathbb{H}^n(s)\}U(s,t_i)\xi ds\right\|_{2}\\
 &\indent+\left\|\int_{t}^TU^*(s,t)\{\mathbb{H}(s)-\mathbb{H}^n(s)\}\{U(s,t)-U(s,t_i)\}\xi ds\right\|_{2}\\
 &\leq \mathcal{C}\sup_{s\in[t,T]}\|\{U(s,t)-U(s,t_i)\}\xi\|_{2}+\mathcal{C}\int_{t_i}^{t_{i+1}}\|\mathbb{H}(s)\|_{\mathcal{L}}ds\\
 &\indent+\mathcal{C}\int_{t_i}^T\|\{\mathbb{H}(s)-\mathbb{H}^n(s)\}U(s,t_i)\xi\|_{2}ds,
 \end{aligned}
\end{equation}
where $\mathbb{H}^n=\Gamma^n\mathbb{H}$. By the choice of $\mathbb{H}^n$, there is an integer $N_2(\varepsilon_1)>0$ so that for all $n>N_2$ and $i=1,\cdots,N_1-1$,
\begin{equation}\label{H-Hn-U-xi}
\int_{t_i}^T\|\{\mathbb{H}(s)-\mathbb{H}^n(s)\}U(s,t_i)\xi\|_{2}ds\leq \varepsilon_1.
\end{equation}
From \eqref{U-tau-sigma}-\eqref{H-Hn-U-xi}, we find that, for any $n>N_2$ and $t\in[t_0,T]$,
\begin{equation*}
\left\|\int_{t}^TU^*(s,t)\{\mathbb{H}(s)-\mathbb{H}^n(s)\}U(s,t)\xi ds\right\|_{2}\leq \mathcal{C}_1\varepsilon_1,
\end{equation*}
where $\mathcal{C}_1$ is independent of $\varepsilon_1$, $n$ and $t$. Similarly, there is an integer $N_3(\varepsilon_1)>0$ such that for every $n>N_3$,
\begin{equation*}
\|U^*(T,t)\{P_T-P^n_T\}U(T,t)\xi\|_{2}\leq \mathcal{C}_2\varepsilon_1,
\end{equation*}
where $\mathcal{C}_2$ is independent of $\varepsilon_1$, $n$ and $t$.
For any $\varepsilon>0$, we choose $\varepsilon_1=\frac{\varepsilon}{\mathcal{C}_1+\mathcal{C}_2}$, then 
\begin{equation*}
\left\|R^{(t)}\xi-P^n(t)\xi\right\|_{2}<\varepsilon,\quad  n>\max\{N_2(\varepsilon_1), N_3(\varepsilon_1)\},\quad  t\in[t_0,T].
\end{equation*}
Therefore, \eqref{the purpose of step 3-2} holds.

To show \eqref{the aim of step 3-1}, for any $0\leq t_1<t_2<T$ and $\eta_1,\eta_2\in L^2(\mathscr{C}_{t_1})$, let $\xi_1=0$, $u_1=\frac{\chi_{[t_1,t_2]}}{t_2-t_1}\eta_1$ and $v_1=0$ in \eqref{QSDE-Forward-1}, and $\xi_2=\eta_2$ and $u_2=v_2=0$ in \eqref{QSDE-Forward-2}. By means of Lemma \ref{the relationship between solution and coefficients} and \eqref{inner product of finally of step2}, we obtain that
\begin{equation}\label{inner product to prove P=R}
\begin{aligned}
&\frac{1}{t_2-t_1}\int_{t_1}^{t_2}\langle P(t)U(t,t_1)\eta_2, \eta_1 \rangle dt\\
=&\left\langle P_TU(T,t_1)\eta_2, x_{1,t_2}(T) \right\rangle-\int_{t_1}^T\langle \mathbb{H}(t)U(t,t_1)\eta_2, x_{1,t_2}(t)\rangle dt,
\end{aligned}
\end{equation}
where $x_{1,t_2}(\cdot)$ is the solution to \eqref{QSDE-Forward-1} with the above choice of $\xi_1$, $u_1$ and $v_1$. Clearly, 
\begin{equation}\label{the solution to x1 under the condition to proof P=R}
x_{1,t_2}(t)=\left\{
\begin{aligned}
&\int_{t_1}^tD(s)x_{1,t_2}(s)ds+\frac{1}{t_2-t_1}\int_{t_1}^t\eta_1ds+\int_{t_1}^tF(s)x_{1,t_2}(s)dW(s),\ &t\in[t_1,t_2],\\
&U(t,t_2)x_{1,t_2}(t_2),\ &t\in[t_2,T].
\end{aligned}
\right.
\end{equation}
By \cite[Theorem 3.5]{B.S.W.1}, together with  \eqref{the definition of M_(D,F)}, we can obtain that for all $t\in[t_1,t_2]$
\begin{equation*}
\|x_{1,t_2}(t)\|_2^2\leq \mathcal{C}\left(\int_{t_1}^tM_{D,F,2}(s)\|x_{1,t_2}(s)\|_{2}^2ds +\|\eta_1\|^2_{2}\right).
\end{equation*}
By the Gronwall inequality, we can derive that
\begin{equation}\label{the estimate x1t1 under the condition to proof P=R}
\|x_{1,t_2}(\cdot)\|^2_{C_\mathbb{A}([t_1,t_2];L^2(\mathscr{C}))}\leq \mathcal{C}\|\eta_1\|^2_{2},
\end{equation}
where the constant $\mathcal{C}$ is independent of $t_2$. On the other hand, by \eqref{the solution to x1 under the condition to proof P=R}, we have
\begin{equation}\label{The estimate of x1t2-eta1}
\begin{aligned}
&\|x_{1,t_2}(t_2)-\eta_1\|^2_{2}
\leq \mathcal{C}\int_{t_1}^{t_2}M_{D,F,2}(s)\|x_{1,t_2}(s)\|_{2}^2ds +\mathcal{C}\left\|\frac{1}{t_2-t_1}\int_{t_1}^{t_2}\eta_1 ds-\eta_1\right\|^2_{2}.
\end{aligned}
\end{equation}
From \eqref{the estimate x1t1 under the condition to proof P=R} and \eqref{The estimate of x1t2-eta1}, we conclude that
\begin{equation*}
\lim_{t_2\to t_1+}\|x_{1,t_2}(t_2)-\eta_1\|^2_{2}\leq \mathcal{C}\lim_{t_2\to t_1+}\left\|\frac{1}{t_2-t_1}\int_{t_1}^{t_2}\eta_1 ds-\eta_1\right\|^2_{2}=0.
\end{equation*}
Therefore, for any $t\in[t_2,T]$,
\begin{align*}
  & \lim_{t_2\to t_1+}\|U(t,t_2)x_{1,t_2}(t_2)-U(t,t_1)\eta_1\|^2_{2}\\
\leq& 2\lim_{t_2\to t_1+}\left(\|U(t,t_2)x_{1,t_2}(t_2)-U(t,t_2)\eta_1\|^2_{2}
  +\|U(t,t_2)\eta_1-U(t,t_1)\eta_1\|^2_{2}\right) \\
\leq&\mathcal{C}\lim_{t_2\to t_1+}\left(\|x_{1,t_2}(t_2)-\eta_1\|^2_{2}+\|U(t,t_2)\eta_1-U(t,t_1)\eta_1\|^2_{2}\right)\\
= &0.
\end{align*}
Hence, for any $t\in[t_2,T]$,
\begin{equation}\label{the formula of x1t2}
 \lim_{t_2\to t_1+}x_{1,t_2}(t)=\lim_{t_2\to t_1+}U(t,t_2)x_{1,t_2}(t_2)=U(t,t_1)\eta_1, \ \textrm{ in}\ L^2(\mathscr{C}_{t}).
\end{equation}
From \eqref{the formula of x1t2}, we can deduce that  
\begin{equation}\label{inner product when t2 to t1}
\begin{aligned}
 & \lim_{t_2\to t_1+}\left\{\langle P_TU(T,t_1)\eta_2, x_{1,t_2}(T)\rangle-\int_{t_1}^{T} \langle \mathbb{H}(t)U(t,t_1)\eta_2,  x_{1,t_2}(t)\rangle dt\right\}\\
 &\indent =\langle P_TU(T,t_1)\eta_2,  U(T,t_1)\eta_1\rangle-\int_{t_1}^{T}\langle \mathbb{H}(t)U(t,t_1)\eta_2, U(t,t_1)\eta_1\rangle dt.
\end{aligned}
\end{equation}
Next, we choose $\xi_1=\eta_1$ and $u_1=v_1=0$ in \eqref{QSDE-Forward-1}, and $\xi_2=\eta_2$ and $u_2=v_2=0$ in \eqref{QSDE-Forward-2}. From \eqref{inner product of finally of step2}, one has that
\begin{equation}\label{inner product of with xi12=u12=v12=0}
\left\langle R^{(t_1)}\eta_2, \eta_1\right\rangle=\left\langle P_TU(T,t_1)\eta_2, U(T,t_1)\eta_1\right\rangle-\int_{t_1}^T\langle \mathbb{H}(t)U(t,t_1)\eta_2, U(t,t_1)\eta_1 \rangle dt.
\end{equation}
From \eqref{inner product to prove P=R}, \eqref{inner product when t2 to t1} and \eqref{inner product of with xi12=u12=v12=0}, we obtain that
\begin{equation}\label{ft}
\lim_{t_2\to t_1+}\frac{1}{t_2-t_1}\int_{t_1}^{t_2}\left\langle P(t)U(t,t_1)\eta_2, \eta_1\right\rangle dt=\left\langle R^{(t_1)}\eta_2, \eta_1\right\rangle.
\end{equation}

On the other hand, by \cite[Lemma 4.14]{L.Z-2020},
there is a monotonically decreasing sequence $\{t_2^{(n)}\}_{n=1}^\infty$ with $t_2^{(n)}>t_1$ for every $n$, such that for any $t_1\in[0,T)$,
\begin{equation}\label{n to infty,P}
 \lim_{t_2^{(n)}\to t_1+}\frac{1}{t_2^{(n)}-t_1}\int_{t_1}^{t_2^{(n)}}\left\langle P(t)U(t,t_1)\eta_2, \eta_1\right\rangle dt=\langle P(t_1)\eta_2, \eta_1\rangle.
\end{equation}
From \eqref{ft} and \eqref{n to infty,P}, we can infer that
\begin{equation*}
\langle R^{(t_1)}\eta_2, \eta_1\rangle=\langle P(t_1)\eta_2, \eta_1\rangle,\ \textrm{ a.e. }  t_1\in[0,T).
\end{equation*}
Since $\eta_1,\eta_2\in L^2(\mathscr{C}_{t_1})$ are arbitrary, we conclude \eqref{the aim of step 3-1}.

It follows from \eqref{inner product of finally of step2} and \eqref{the aim of step 3-1} that $ \left(P(\cdot),Q^{(\cdot)},\widehat{Q}^{(\cdot)}\right)$ satisfies \eqref{Def of relaxed transposition solution}. Hence, the triple $\left(P(\cdot),Q^{(\cdot)},\widehat{Q}^{(\cdot)}\right)$ is a relaxed transposition solution to \eqref{BSDE-P}, and by \eqref{the estimate about P}, \eqref{the estimate of Q1t0}, \eqref{the estimate of Q2-t0}, \eqref{the estimate of Q3-t0} and \eqref{definition of Qt0 and widetilde Qt0}, we derive that \eqref{estimate of main thm} holds. Then we complete the proof of the existence of the relaxed transposition solution.

Finally, we prove the uniqueness of the relaxed transposition solution to \eqref{BSDE-P}. Assume that both $ \left(P(\cdot),Q^{(\cdot)},\widehat{Q}^{(\cdot)}\right)$ and $\left(\overline{P}(\cdot),\overline{Q}^{(\cdot)},\widehat{\overline{Q}}^{(\cdot)}\right)$ are two relaxed transposition solutions to \eqref{BSDE-P}. Then, from \eqref{Def of relaxed transposition solution}, for any $t_0\in[0,T]$, we have
\begin{equation}\label{inner product of uniqueness of the transposition solution}
\begin{aligned}
 0=&\langle\{P(t_0)-\overline{P}(t_0)\}\xi_2, \xi_1\rangle+\int_{t_0}^T\langle \{P(t)-\overline{P}(t)\}F(t)x_2(t),v_1(t)\rangle dt\\
 &+\int_{t_0}^T\langle\{P(t)-\overline{P}(t)\}u_2(t), x_1(t)\rangle dt+\int_{t_0}^T\langle \{P(t)-\overline{P}(t)\}x_2(t), u_1(t)\rangle dt\\
&+\int_{t_0}^T\langle \{P(t)-\overline{P}(t)\}v_2(t), F(t)x_1(t)+v_1(t)\rangle dt\\
&+\int_{t_0}^T\left\langle Q^{(t_0)}(\xi_2,u_2,v_2)(t)-\overline{Q}^{(t_0)}(\xi_2,u_2,v_2)(t), v_1(t)\right\rangle dt\\
&+\int_{t_0}^{T}\left\langle v_2(t), \widehat{Q}^{(t_0)}(\xi_1,u_1,v_1)(t)-\widehat{\overline{Q}}^{(t_0)}(\xi_1,u_1,v_1)(t) \right\rangle dt.
\end{aligned}
\end{equation}

Let $u_1=u_2=0$ and $v_1=v_2=0$ respectively in \eqref{QSDE-Forward-1} and \eqref{QSDE-Forward-2}. From \eqref{inner product of uniqueness of the transposition solution}, we obtain that
\begin{equation*}
\langle \{P(t_0)-\overline{P}(t_0)\}\xi_2,\ \xi_1\rangle=0, \quad \xi_1,\xi_2\in L^2(\mathscr{C}_{t_0}).
\end{equation*}
Hence, $P(\cdot)=\overline{P}(\cdot)$. It is easy to see that \eqref{inner product of uniqueness of the transposition solution} becomes that
\begin{equation*}
\begin{aligned}
0=&\int_{t_0}^T\left\langle Q^{(t_0)}(\xi_2,u_2,v_2)(t)-\overline{Q}^{(t_0)}(\xi_2,u_2,v_2)(t),\ v_1(t)\right\rangle dt\\
&+\int_{t_0}^T\left\langle v_2(t),\ \widehat{ Q}^{(t_0)}(\xi_1,u_1,v_1)(t)-\widehat{\overline{Q}}^{(t_0)}(\xi_1,u_1,v_1)(t)\right\rangle dt,\ t_0\in[0,T].
\end{aligned}
\end{equation*}
Choosing $v_1=0$ in \eqref{QSDE-Forward-1}, then the above equation becomes
\begin{equation*}
0=\int_{t_0}^T\left\langle v_2(t),\ \widehat{Q}^{(t_0)}(\xi_1,u_1,0)(t)-\widehat{\overline{Q}}^{(t_0)}(\xi_1,u_1,0)(t)\right\rangle dt.
\end{equation*}
By the arbitrary of $v_2$ in $L_\mathbb{A}^2([0,T];L^2(\mathscr{C}))$, we deduce that $\widehat{Q}^{(t_0)}(\cdot,\cdot,0)=\widehat{\overline{Q}}^{(t_0)}(\cdot,\cdot,0)$. Similarly, $Q^{(t_0)}(\cdot,\cdot,0)=\overline{Q}^{(t_0)}(\cdot,\cdot,0)$. Hence,
\begin{equation}\label{inner product of unqiue with Q}
\begin{aligned}
0=&\int_{t_0}^T\left\langle Q^{(t_0)}(0,0,v_2)(t)-\overline{Q}^{(t_0)}(0,0,v_2)(t), v_1(t)\right\rangle dt\\
&+\int_{t_0}^T\left\langle v_2(t),\widehat{Q}^{(t_0)}(0,0,v_1)(t)-\widehat{\overline{Q}}^{(t_0)}(0,0,v_1)(t)\right\rangle dt.
\end{aligned} 
\end{equation}
Since $\overline{Q}^{(t_0)}(0,0,v)^*=\widehat{\overline{Q}}^{(t_0)}(0,0,v)$ and $Q^{(t_0)}(0,0,v)^*=\widehat{Q}^{(t_0)}(0,0,v)$ for any $v\in L_\mathbb{A}^2([0,T];L^2(\mathscr{C}))$,
together with \eqref{inner product of unqiue with Q}, we have that
\begin{equation*}
 0=2\int_{t_0}^T\left\langle v_2(t),\widehat{Q}^{(t_0)}(0,0,v_1)(t)-\widehat{\overline{Q}}^{(t_0)}(0,0,v_1)(t)\right\rangle dt.
\end{equation*}
This implies that $Q^{(t_0)}(0,0,\cdot)=\overline{Q}^{(t_0)}(0,0,\cdot)$ and $\widehat{Q}^{(t_0)}(0,0,\cdot)=\widehat{\overline{Q}}^{(t_0)}(0,0,\cdot)$. Hence $Q^{(t_0)}(\cdot,\cdot,\cdot)=\overline{Q}^{(t_0)}(\cdot,\cdot,\cdot)$ and $\widehat{Q}^{(t_0)}(\cdot,\cdot,\cdot)=\widehat{\overline{Q}}^{(t_0)}(\cdot,\cdot,\cdot)$ for any $t_0\in [0,T]$. Therefore, the relaxed transposition solution to \eqref{BSDE-P} is unique.
\end{proof}

\end{document}